\documentclass[12pt]{amsart}
\usepackage{pgf,tikz,pgfplots} %For TikZ
\pgfplotsset{compat=1.14}
\usepackage{mathrsfs}
\usetikzlibrary{arrows}
\usetikzlibrary{patterns}
\usepackage{pgfplots}
\usetikzlibrary{intersections, pgfplots.fillbetween}
\usepackage[colorlinks=true,urlcolor=blue, citecolor=red,linkcolor=blue,linktocpage,pdfpagelabels, bookmarksnumbered,bookmarksopen]{hyperref}
\usepackage[hyperpageref]{backref}
\usepackage{cleveref}
\usepackage{xcolor}
\usepackage{amsthm} %for citing inside theorem header
\usepackage{latexsym,amsmath,amssymb}
\usepackage{accents}
\usepackage[colorinlistoftodos,prependcaption,textsize=tiny]{todonotes}
\usepackage{a4wide}
\usepackage{soul}
\usepackage{mathtools} %For \chi with a much lower index (\chi)
\usepackage{xparse} %For a new definiton of \chi with a slightly lower index
\usepackage{enumitem}

\pgfdeclarelayer{ft}
\pgfdeclarelayer{bg}
\pgfsetlayers{bg,main,ft}

\title{A remark on staircase laminates in restricted sets}

\author[I. Buchowiec]{Igor Buchowiec}
\address[Igor Buchowiec]{
Doctoral School of Exact and Natural Sciences,%
University of Warsaw,
Banacha 2,
02-097 Warszawa, Poland}
\email{i.buchowiec@uw.edu.pl}
\author[P. Kamthorntaksina]{Pholphum Kamthorntaksina}
\address[Pholphum Kamthorntaksina]{
Department of Mathematics, 
Chulalongkorn University, 
Bangkok,
Thailand}
\email{pholphumk@gmail.com}
\author[K. Mazowiecka]{Katarzyna Mazowiecka}
\address[Katarzyna Mazowiecka]{
Institute of Mathematics,%
University of Warsaw,
Banacha 2,
02-097 Warszawa, Poland}
\email{k.mazowiecka@mimuw.edu.pl}
\author[A. Schikorra]{Armin Schikorra}
\address[Armin Schikorra]{Department of Mathematics,
 University of Pittsburgh,
 301 Thackeray Hall,
 Pittsburgh, PA 15260, USA \and Department of Mathematics, 
Chulalongkorn University, 
Bangkok,
Thailand}
  \email{armin@pitt.edu}

\author[A. Vincent]{Akshara Vincent}
\address[Akshara Vincent]{Department of Mathematics,
 University of Pittsburgh,
 301 Thackeray Hall,
 Pittsburgh, PA 15260, USA}
 \email{akv20@pitt.edu}
\definecolor{indigo}{rgb}{0.29, 0.0, 0.51}
\definecolor{p1}{gray}{0.4}
\definecolor{p2}{gray}{0.6}
\definecolor{p3}{gray}{0.98}
\definecolor{p4}{gray}{0.8}
\definecolor{p5}{gray}{0.9}

\newcommand{\Ka}{\mathscr{U}_{\bar{a},\delta,\gamma}}
\newcommand{\la}{\Lambda}

plus 6pt minus 12pt
\def\eps{\varepsilon}

\def\B{\mathbb{B}}

\def\N{{\mathbb N}}

\def\S{{\mathbb S}}

\renewcommand{\div}{{\rm div}}

\newtheorem{theorem}{Theorem}
\newtheorem{lemma}[theorem]{Lemma}

\newtheorem{proposition}[theorem]{Proposition}

\newtheorem{definition}[theorem]{Definition}

\def\rank{{\rm rank\,}}
\def\supp{{\rm supp\,}}

\newcommand{\R}{\mathbb{R}}

\newcommand{\brac}[1]{\left (#1 \right )}
\newcommand{\abs}[1]{\left |#1 \right |}

\newcommand{\barint}{
\rule[.036in]{.12in}{.009in}\kern-.16in \displaystyle\int }

\newcommand{\barcal}{\mbox{$ \rule[.036in]{.11in}{.007in}\kern-.128in\int $}}

\newcommand*\oline[1]{%
  \kern0.1em            % Counteract the inner kern
  \vbox{%
    \hrule height 0.5pt % Line above with certain width
    \kern0.4ex          % Distance between line and content
    \hbox{%
      \kern-0.1em       % Shorten the content on the left
      $#1$%             % The content, typeset math mode
      \kern-0.1em       % Shorten the content on the right
    }% end of hbox
  }% end of vbox
  \kern0.1em            % Counteract the inner kern
}

\def\mvint_#1{\mathchoice
          {\mathop{\vrule width 6pt height 3 pt depth -2.5pt
                  \kern -8pt \intop}\nolimits_{\kern -3pt #1}}%
          {\mathop{\vrule width 5pt height 3 pt depth -2.6pt
                  \kern -6pt \intop}\nolimits_{#1}}%
          {\mathop{\vrule width 5pt height 3 pt depth -2.6pt
                  \kern -6pt \intop}\nolimits_{#1}}%
          {\mathop{\vrule width 5pt height 3 pt depth -2.6pt
                  \kern -6pt \intop}\nolimits_{#1}}}

\numberwithin{theorem}{section} \numberwithin{equation}{section}

\newcommand{\aleq}{\precsim}
\newcommand{\ageq}{\succsim}
\newcommand{\aeq}{\asymp}

\let\latexchi\chi
\makeatletter
\renewcommand\chi{\@ifnextchar_\sub@chi\latexchi}
\newcommand{\sub@chi}[2]{% #1 is _, #2 is the subscript
  \@ifnextchar^{\subsup@chi{#2}}{\latexchi_{#2}}%
}
\newcommand{\subsup@chi}[3]{% #1 is the subscript, #2 is ^, #3 is the superscript
  \latexchi_{#1}^{#3}%
}
\makeatother
\makeatletter
\def\tikz@arc@opt[#1]{% over-write!
  {%
    \tikzset{every arc/.try,#1}%
    \pgfkeysgetvalue{/tikz/start angle}\tikz@s
    \pgfkeysgetvalue{/tikz/end angle}\tikz@e
    \pgfkeysgetvalue{/tikz/delta angle}\tikz@d
    \ifx\tikz@s\pgfutil@empty%
      \pgfmathsetmacro\tikz@s{\tikz@e-\tikz@d}
    \else
      \ifx\tikz@e\pgfutil@empty%
        \pgfmathsetmacro\tikz@e{\tikz@s+\tikz@d}
      \fi%
    \fi
    \tikz@arc@moveto
    \xdef\pgf@marshal{\noexpand%
    \tikz@do@arc{\tikz@s}{\tikz@e}
      {\pgfkeysvalueof{/tikz/x radius}}
      {\pgfkeysvalueof{/tikz/y radius}}}%
  }%
  \pgf@marshal%
  \tikz@arcfinal%
}
\let\tikz@arc@moveto\relax
\def\tikz@arc@movetolineto#1{%
  \def\tikz@arc@moveto{\tikz@@@parse@polar{\tikz@arc@@movetolineto#1}(\tikz@s:\pgfkeysvalueof{/tikz/x radius} and \pgfkeysvalueof{/tikz/y radius})}}
\def\tikz@arc@@movetolineto#1#2{#1{\pgfpointadd{#2}{\tikz@last@position@saved}}}
\tikzset{%
  move to start/.code=\tikz@arc@movetolineto\pgfpathmoveto,%
  line to start/.code=\tikz@arc@movetolineto\pgfpathlineto}
\makeatother
\begin{document}
\begin{abstract}
We slightly extend the convex integration via staircase laminate toolbox recently developed by Kleiner, M\"uller, Sz\'{e}kelyhidi, and Xie. As an example we revisit the proof by Astala--Faraco--Sz\'{e}kelyhidi on optimal Meyers' regularity theory via this framework.
\end{abstract}
\maketitle
\tableofcontents

\section{Introduction}

In \cite{KMSX24}, Kleiner, M\"uller, Sz\'{e}kelyhidi, and Xie provided a very useful framework for convex integration via laminates, introducing the notion of $L^p$-reducibility between sets. As applications, they provided streamlined proofs for several results, including the counterexample of Colombo--Tione \cite{ColomboTione} to the conjecture of Iwaniec--Sbordone \cite{IwaniecSbordone94}.

Notably missing and explicitly excluded from these examples, cf. \cite[p.1574]{KMSX24}, is the following classical result by Astala--Faraco--Sz\'{e}kelyhidi \cite{AFS08}:

Whenever $A\colon \B^m \to \R^{m \times m}$ is a symmetric, bounded, measurable, uniformly elliptic matrix, i.e., $\frac{1}{\Lambda} I \leq A \leq \Lambda I$ a.e. in $\B^m$, then Gehring's lemma provides Meyers improved regularity result: a solution $u \in W^{1,2}_0(\B^m)$ to $\div(A\nabla u) = 0$ in $\B^m$ satisfies $\nabla u \in L^{2+\eps}(\B^m)$. The precise size of $\eps$ (in dependence on the ellipticity constant) is unknown in general. The following result provides an upper bound for $\eps$, which is sharp in two dimensions, see \cite{A94,LN97,PV02}.

\begin{theorem}[\cite{AFS08}]\label{th:AFS}
For any $\Lambda > 1$ there exists a bounded measurable symmetric matrix $A\colon  \B^2 \to \R^{2 \times 2}$, with
\[
\frac{1}{\Lambda} I \leq  A \leq \Lambda I
\]
a function $u \in W^{1,2}(\B^2, \R)$, and an affine linear map $l \colon  \R^2\to\R$ such that
\[
\begin{cases}
  \div (A \nabla u) = 0 \quad &\text{in $\B^2$}\\
  u(x) = {l}\quad &\text{on $\partial \B^2$}
\end{cases}
\]
and \[\int_{\B^2} |\nabla u|^{p} = \infty\] for any $p \geq \frac{2\Lambda}{\Lambda -1}$.
\end{theorem}

In this note, we extend the framework of \cite{KMSX24} to include \Cref{th:AFS}.

To this end, we introduce, for open sets $U \subset \R^{d\times m}$, the notion of laminate \emph{with construction steps in $U$} (see \Cref{def:laminatewithconstructionsteps}) and the notion of $U$-reduced in weak $L^p$ (see \Cref{reduc}).

The following result says that if $U$ can be $U$-reduced to $K$ in weak $L^p$ then we can find many exact solutions to the differential inclusion $\nabla w \in K$. This is a slight extension of \cite[Theorem 4.1]{KMSX24}, which is \Cref{th:thm4.1} for $U = \R^{d \times n}$.

\begin{theorem}[Exact solution, cf. {\cite[Theorem 4.1]{KMSX24}}]\label{th:thm4.1}
Let $K\subset \mathbb{R}^{d\times m}$, $U \subset \R^{d\times m}$ open, and $1<p<\infty$ such that
$U$ can be {$U$}-reduced to $K$ in weak $L^{p}$, cf. \Cref{reduc}.

% Let $M,m$ be the constant obtained by the reducibility.
Then for any regular domain $\Omega\subset \R^m $, any
$X_0\in U$, $b\in \mathbb{R}^{d}$ and any
$\delta>0$, $\alpha\in(0,1)$ there exists a piecewise affine map
$w\in W^{1,1}(\Omega,\R^d)\cap C^{\alpha}(\bar\Omega,\R^d)$ with
$w(x)=l_{X_0,b}(x)\equiv X_0x+b$ on $\partial\Omega$ such that
\[
\nabla w(x)\in K \quad \text{a.e. } x\in \Omega
\]
\[
\|w-l_{X_0,b}\|_{C^{\alpha}(\bar\Omega,\R^d)}<\delta
\]
and there exists a constant $C>0$ such that for all $t>0$,
\begin{equation}\label{eq:wewant4.2c}
\bigl|\{x\in\Omega\colon \ |\nabla w(x)|>t\}\bigr|
\leq  C t^{-p}. 
\end{equation}
If moreover, $U$ can be {$U$}-reduced to $K$ {\emph{exactly}} in weak $L^{p}$ then for another constant $c>0$,
\begin{equation}\label{eq:wewant4.2c2}
\bigl|\{x\in\Omega\colon \ |\nabla w(x)|>t\}\bigr|
\geq c t^{-p} \quad \forall t \in (1,\infty).
\end{equation}

\end{theorem}

In order to obtain the assumptions of \Cref{th:thm4.1} we have a laminate criterion, which is a slight adaptation of \cite[Theorem 4.3]{KMSX24}.

\begin{theorem}[Staircase laminate criterion, cf. {\cite[Theorem 4.3]{KMSX24}}]\label{th:thm4.3}
Let $K\subset \mathbb{R}^{d\times m}$, $U\subset \R^{d\times m}$ open, and $1<p<\infty$. Assume that there exist constants $M\ge 1$ and $m>0$ with the following property: for any $X_0\in U$ there exists a staircase laminate $\nu_{X_0}^\infty$ supported on $K$, with barycenter $X_0$ and {construction steps in $U$}, see \Cref{defstair}, satisfying the bound
\begin{equation} \label{eq1.1}
\nu_{X_0}^\infty\bigl(\{X\colon \ |X|>t\}\bigr)\le M^p\,(1+|X_0|^p)\,t^{-p}
\qquad\text{for all } t>0.
\end{equation}
Then $U$ can be $U$-reduced to $K$ in weak $L^p$.

If, in addition to \eqref{eq1.1}, we also have
\begin{equation}\label{eq:1.1lower}
\nu_{X_0}^\infty\bigl(\{X\colon \ |X|>t\}\bigr)\geq m^p\,t^{-p}
\qquad\text{for all } t\in (1,\infty),
\end{equation}
then $U$ can be $U$-reduced to $K$ {exactly} in weak $L^p$.
% Assume for any $A \in U$ there exists a staircase laminate \emph{with construction steps in $U$} supported in $K$, with barycenter ... bound. Then $U$ can be $U$-reduced to $K$ in weak $L^p$.
\end{theorem}

{\bf Outline.}
In \Cref{s:defReduce} we introduce the notion of $U$-reducibility. We then adapt the proofs in \cite{KMSX24} to establish \Cref{th:thm4.1}. In \Cref{s:laminatecrit} we introduce the notion of construction steps in $U$ and prove \Cref{th:thm4.3}. Finally, in \Cref{s:AFS} we apply this framework to obtain a streamlined proof of \Cref{th:AFS} using $L^p$-reducibility.

{\bf Notation.}
We use the standard notation $\aleq$, $\aeq$, and $\ageq$. We write $A \aleq B$ if there exists a constant $C>0$, depending only on irrelevant data, such that
\[
A \le C\,B.
\]
We write $A \aeq B$ if $A \aleq B$ and $B \aleq A$.

{\bf Acknowledgment.} The project is co-financed by
\begin{itemize}
 \item (KM) the Polish National Agency for Academic Exchange within Polish Returns Programme -
BPN/PPO/2021/1/00019/U/00001;
\item (KM, IB) the National Science Centre, Poland grant No. 2023/51/D/ST1/02907;
\item(AS) the NSF Career DMS-2044898.
\end{itemize}
Armin Schikorra appreciates the hospitality of the Thematic Research Programme Geometric Analysis: Methods and Applications at the University of Warsaw, where part of this work was conducted. 
Armin Schikorra and Pholphum Kamthorntaksina's research was supported in part by the International Centre for Theoretical Sciences (ICTS) for participating in the program - Geometric Analysis and PDE 2026 (code: ICTS/GPDE2026/02). 
Armin Schikorra was  
Alexander von Humboldt Research Fellow.

\section{Notion of reducibility and proof of Theorem~\ref{th:thm4.1}}\label{s:defReduce}
We begin by defining the $U$-reducibility, which essentially just means that the laminate construction process can restart from within $U$. We also recall the following two standard definitions.
\begin{definition}[Regular domain]
A set $\Omega \subset \mathbb{R}^m$ is called a \emph{regular domain} if it is open, bounded, connected, and $\partial\Omega$ has zero $m$-dimensional Lebesgue measure.
\end{definition}

\begin{definition}[Piecewise affine map]\label{piecewiseaf}
Let $\Omega\subset\R^m$ be a regular domain. We call a map $w\in W^{1,1}(\Omega,\R^d)$ \emph{piecewise affine} if there exists an at most countable decomposition $\Omega=\bigcup_i \Omega_i\cup\mathcal{N}$ into pairwise disjoint regular domains $\Omega_i\subset\Omega$ and a null set $\mathcal{N}$ such that $w$ agrees with an affine map on each $\Omega_i$. That is, for any $i$ there exists $X_i\in\R^{d\times m}$ and $b_i\in\R^d$ such that $w(x)=X_ix+b_i$ for all $x\in\Omega_i$. We also denote $\mathring\Omega=\bigcup_i \Omega_i$ the open subset of $\Omega$ where $u$ is locally affine. Note that regular domains $\Omega_i$ are exactly the connected components of $\mathring\Omega$.
\end{definition}

The main difference of the following notion of reducibility in $U$ to the original version in \cite{KMSX24} is the assumption \eqref{eq2.3}.
\begin{definition}[Reducibility and exact reducibility in weak $L^p$]\label{reduc}
For $U,K\subset\R^{d\times m}$ and $1<p<\infty$ we say that
\[\emph{$U$ can be $U$-reduced to $K$ in weak $L^p$}\]
provided there exists a constant $M=M(U,K,p)\geq 1$ satisfying for an arbitrary $X_0\in U$, $b\in\R^d$, $\eps,\alpha\in (0,1)$, $q\in(1,\infty)$, and a regular domain $\Omega\subset\R^m$ the following property:

There exists a piecewise affine map $w\in W^{1,1}(\Omega,\R^d)\cap C^\alpha(\bar\Omega,\R^d)$ with $w(x)=X_0x+b$ on $\partial\Omega$ such that 
\begin{equation} \label{eq2.2lower}
    |\{x\in\Omega\colon |\nabla w(x)|>t\}|\leq M^p(1+|X_0|^p)|\Omega|t^{-p} \; \text{ for all }t>0,
\end{equation}
and on the error set
\[\Omega_{err}\coloneqq\{x\in\Omega\colon \nabla w(x)\notin K\}\]
we still have
\begin{equation} \label{eq2.3}
    \nabla w(x) \in U \quad \text{for a.e. } x\in \Omega_{err},
\end{equation}
and
\begin{equation} \label{eq2.1}
    \int_{\Omega_{err}}(1+|\nabla w|)^q dx<\eps|\Omega|.
\end{equation}

If in addition to \eqref{eq2.2lower}, there also exists a constant $m=m(U,K,p)>0$ (in particular independent of $\eps>0$) so that
\begin{equation}\label{eq:exactreducibility}
    |\{x\in\Omega\colon |\nabla w(x)|>t\}|\geq m^p
    % (1+|X_0|^p)
    |\Omega|t^{-p} \; \text{ for all }t>{1}, 
\end{equation}
then we say that $U$ can be $U$-reduced to $K$ \emph{exactly} in weak $L^p$.
\end{definition}

\Cref{th:thm4.1} is then a minor adaptation of \cite[Theorem~4.1]{KMSX24}.

We begin with the following lemma. 
\begin{lemma}\label{la:exactPART}
Let $K\subset \mathbb{R}^{d\times m}$, let $U \subset \R^{d\times m}$ be open, and let $1<p<\infty$ be such that
$U$ can be {$U$}-reduced to $K$ in weak $L^{p}$ (in sense of \Cref{reduc}).

Fix $X_0 \in U$ and set $l_{X_0,b} \coloneqq  X_0 x + b$. Let $\Omega \subset \R^m$ be a regular domain. Fix parameters $\delta >0$ and $\alpha \in (0,1)$.

Then, we can find a sequence of piecewise affine maps $w_k\in W^{1,1}(\Omega,\R^d)\cap C^\alpha(\overline{\Omega},\R^d)$, for $k \in \N$, such that $w_k=l_{X_0,b}$ on $\partial\Omega$ and the following hold.

\begin{enumerate}
 \item Define
\[
\Omega_{k,err}
\coloneqq
\{x\in\Omega \colon \nabla w_k(x)\notin K\}.
\]
Then
\begin{equation} \label{eq2.17}
\int_{\Omega_{k,err}} (1+|\nabla w_k|)^p \, dx
< 2^{-k}|\Omega|.
\end{equation}

\item For $k \ge 2$,
\[
\Omega_{k,err}
\subset
\Omega_{k-1,err} \cup N
\]
for some null set $N$, and
\[
\nabla w_k = \nabla w_{k-1}
\quad \text{a.e. on }
\Omega \setminus \Omega_{k-1,err}.
\]
\item \begin{equation} \label{eq2.18}
\|w_k-l_{X_0,b}\|_{C^\alpha(\overline{\Omega})}
< \delta(1-2^{-k}).
\end{equation}
\item For every $t>0$,
\begin{equation} \label{eq2.19}
|\{x\in\Omega : |\nabla w_k(x)|>t\}|
\le
M^p(1+|X_0|^p)|\Omega|t^{-p}
\sum_{i=0}^{k-1} 2^{-i}.
\end{equation}
\item \begin{equation} \label{eq:dust}
\nabla w_k(x) \in U
\quad \text{for a.e. } x\in \Omega_{k,err}.
\end{equation}
\end{enumerate}

{ If we have exact reducibility, then additionally
 \begin{equation} \label{eq2.19exact}
     |\{x\in\Omega \setminus \Omega_{k,err}\colon |\nabla w_k(x)|>t\}|\geq \frac{1}{100} m^p |\Omega|t^{-p}  \quad \text{for any } t \in (1,\infty).
 \end{equation}
 }
\end{lemma}
\begin{proof}[Proof of \Cref{la:exactPART}]
\underline{To construct the map $w_1$}, we consider the regular domain $\Omega_0 = (0,1)^m \subset \R^m$.

 Since $U$ can be $U$-reduced to $K$ in weak $L^p$, by \Cref{reduc}, there exists a piecewise affine map $f\in W^{1,1}(\Omega_0,\R^d)\cap C^\alpha(\overline{\Omega_0},\R^d)$ such that $f=l_{X_0,b}$ on $\partial\Omega_0$ and for some fixed ${\lambda} \in (0,1)$
 \begin{equation} \label{eq2.20}
     \int_{\Omega_{0,err}} (1+|\nabla f|)^p dx < \frac{1}{2} {\lambda}  |\Omega_0|,
 \end{equation}
 and
 \begin{equation} \label{eq2.21}
     |\{x\in\Omega_0 \colon |\nabla f(x)|>t\}| \le M^p(1+|X_0|^p)|\Omega_0|t^{-p} \quad \text{for any } t>0,
 \end{equation}
 and 
 \begin{equation}
     \nabla f(x)\in U \quad \text{for a.e. }x\in \Omega_{0,err},
 \end{equation}
 where $\Omega_{0,err} \coloneqq  \{x\in\Omega_0 : \nabla f(x)\notin K\}$, and 
\[
\|f-l_{X_0,b}\|_{C^{\alpha}(\Omega_0)} < \infty.
\]

{If we have \emph{exact} reducibility, in addition we have 
 \begin{equation} \label{eq2.21exact}
     |\{x\in\Omega_0 \colon |\nabla f(x)|>t\}| \geq m^p |\Omega_0|t^{-p} \quad \text{for any } t \in (1,\infty).
 \end{equation}}

Fix some small $\eps' \in (0,1)$ (to be chosen later). By covering, we obtain countable sequences $r_i\in (0,1)$ and $x_i\in\R^{d\times m}$ such that $r_i^{1-\alpha}\le\varepsilon'$, and the countable disjoint family of regular domains $\{\Omega_i\}_{i\in\N}$, $\Omega_i\coloneqq r_i\Omega_0+x_i$, such that $\big|\Omega\setminus\bigcup_i \Omega_i\big|=0$.
Set 
\[
 f_i(x) \coloneqq  r_i\brac{f\left(\frac{x-x_i}{r_i}\right) - b}+l_{X_0,b}
 \]
 Observe that on $\partial \Omega_i$ we have $f_i\Big|_{\partial \Omega_i}(x) = l_{X_0,b}$.
 
 % Applying the rescaling-covering argument (cf. \cite[Section 2.1]{KMSX24}) with the domain $\Omega$ and $\varepsilon'>0$ (will be chosen later), we obtain $r_i\in (0,1)$ and $x_i\in\R^{d\times m}$ such that $r_i^{1-\alpha}\le\varepsilon'$, and a countable disjoint family of regular domains $\{\Omega_i\}_{i\in\N}$, $\Omega_i=r_i\Omega_0+x_i$, such that $\big|\Omega\setminus\bigcup_i \Omega_i\big|=0$, together with the map 
 % \[
 % f_i(x) = r_if\left(\frac{x-x_i}{r_i}\right)+X_0x_i+(1-r_i)b
 % \]
 % on each $\Omega_i$. Then $f_i = l_{X_0,b}$ on $\partial\Omega_i$ and, 
 We have
 \[
 \|f_i-l_{X_0,b}\|_{C^\alpha(\overline{\Omega_i})}\aleq r_i^{1-\alpha}\|f-l_{X_0,b}\|_{C^\alpha(\overline{\Omega_0})}\le \varepsilon'\|f-l_{X_0,b}\|_{C^\alpha(\overline{\Omega_0})}.
 \]
 Moreover, we observe that $\nabla f_i(x) = \nabla f\left(\frac{x-x_i}{r_i}\right)$ on $\Omega_i$, so
 \[
 \Omega_{i}^{err} \coloneqq  \{x\in\Omega_i : \nabla f_i(x)\notin K\} = r_i\Omega_{0,err}+x_i,
 \]
 and
 \begin{equation}\label{eq:subdust}
     \nabla f_i(x) \in U \quad \text{for a.e. } x\in\Omega_i^{err}.
 \end{equation}

 Define $w_1 : \Omega\to\R^d$ by
 \begin{equation*}
     w_1(x) = \begin{cases}
         f_i(x) & \text{if } x\in\Omega_i ; \\
         l_{X_0,b}(x) & \text{if } x\in\Omega\setminus\bigcup_i \Omega_i.
     \end{cases}
 \end{equation*}
Clearly,  \eqref{eq:dust} is satisfied by \eqref{eq:subdust}.

Also, by the gluing argument (cf. \cite[Section 2.1]{KMSX24}), we obtain a piecewise affine map $w_1\in W^{1,1}(\Omega,\R^d)\cap C^\alpha(\overline{\Omega},\R^d)$ such that $w_1=l_{X_0,b}$ on $\partial\Omega$. Moreover, 
\begin{equation}
    \Omega_{1,err} = \bigcup_i \Omega_i^{err}\cup N,
\end{equation}
where $N=\{x\in\Omega\setminus\bigcup_i\Omega_i : \nabla l_{X_0,b}(x) = X_0\notin K\}$ has measure zero.

 By changing variables $x=r_iy+x_i$ and using \eqref{eq2.20}, we obtain
 \begin{equation*}
 \begin{split}
     \int_{\Omega_{1,err}} (1+|\nabla w_1|)^pdx = \sum_i \int_{\Omega_i^{err}} (1+|\nabla f_i|)^pdx 
     &= \sum_i \int_{\Omega_{0,err}} (1+|\nabla f|)^p r_i^m dy \\
     &< \sum_i \frac{1}{2} r_i^m|\Omega_0| \\
     &= \sum_i \frac{1}{2}|\Omega_i| \\
     &= \frac{1}{2}|\Omega|,
 \end{split}
 \end{equation*}
 which is \eqref{eq2.17} for $w_1$.

 Pick $\varepsilon'>0$ small enough so that $2\varepsilon'\|f-l_{X_0,b}\|_{C^\alpha(\overline{\Omega_0},\R^d)}<\delta/2$. By a gluing argument (cf. \cite[Section 2.1, p.1575]{KMSX24}), we obtain \eqref{eq2.18} for $w_1$. Indeed,
 \[
 \|w_1-l_{X_0,b}\|_{C^\alpha(\overline{\Omega},\R^d)}\le 2\sup_i \|f_i-l_{X_0,b}\|_{C^\alpha(\overline{\Omega_i},\R^d)} \le 2\varepsilon'\|f-l_{X_0,b}\|_{C^\alpha(\overline{\Omega_0},\R^d)}<\frac{\delta}{2}.
 \]

 Next, since $\Omega_i$ are pairwise disjoint we find 
 \[
 |\{x\in\Omega : |\nabla w_1(x)|>t\}| = \sum_{i=1} |\{x\in\Omega_i : |\nabla f_i(x)|>t\}|.
 \]
 We also observe that, for each $i$,
 \[
 \{x\in\Omega_i : |\nabla f_i(x)|>t\} = r_i\{x\in\Omega_0 : |\nabla f(x)|>t\}+x_i.
 \]
 and thus 
 \[
 |\{x\in\Omega : |\nabla w_1(x)|>t\}| = \sum_{i=1} r_i^m |\{x\in\Omega_0 : |\nabla f(x)|>t\}|.
 \]
 From \eqref{eq2.21} we find, using again that $\Omega$ is the disjoint union of the $\Omega_i$ up to a set of measure zero,
 \begin{equation*}
     \begin{split}
         |\{x\in\Omega : |\nabla w_1(x)|>t\}| &\le \sum_i r_i^m M^p(1+|X_0|^p)|\Omega_0|t^{-p} \\
         &= \sum_i M^p(1+|X_0|^p)|\Omega_i|t^{-p} \\
         &=M^p(1+|X_0|^p)|\Omega|t^{-p}.
     \end{split}
 \end{equation*}
That is we have established \eqref{eq2.19} for $w_1$.

 In order to establish \eqref{eq2.19exact}, we argue similarly
 \begin{equation*}
     \begin{split}
         |\{x\in\Omega : |\nabla w_1(x)|>t\}| &\geq \sum_i r_i^m m^p |\Omega_0|t^{-p} \\
         &= \sum_i m^p(1+|X_0|^p)|\Omega_i|t^{-p} \\
         &=m^p(1+|X_0|^p)|\Omega|t^{-p}.
     \end{split}
 \end{equation*}

On the other hand, by \eqref{eq2.20} and Chebyshev inequality
 \[
 |\{x\in\Omega_{err,1} : |\nabla w_1(x)|>t\}| \leq t^{-p} \frac{{\lambda}}{2} |\Omega|.
 \]
 For ${\lambda}$ suitably small (depending only on $m^p$) we then have 
 \[
|\{x\in\Omega \setminus \Omega_{err,1}: |\nabla w_1(x)|>t\}| \geq \brac{m^p - \frac{\lambda}{2}} |\Omega|t^{-p} \geq \frac{1}{100} m^p t^{-p} |\Omega|.
 \]
This gives \eqref{eq2.19exact}, and we conclude the construction of $w_1$. 

 \underline{For the inductive step}, fix some $k$ and assume we have found $w_k$. Since $w_k$ is piecewise affine, we can decompose countably $\Omega_{k,err} = \bigcup_j \Omega_j \cup \mathcal{N}$ such that $\mathcal{N}$ has measure zero and $w_k = l_{X_j,b_j}$ on each $\Omega_j$. 
 
 From the condition \eqref{eq:dust}, we have $X_j\in U$ for any $j$. We then apply the \Cref{reduc} with $\Omega_0=(0,1)^m, X_j\in U, b_j\in\R^d$ and $\varepsilon_0=2^{-(k+1)}$ and then apply the rescaling-covering and gluing arguments as above to obtain the piecewise affine map $g_j\in W^{1,1}(\Omega_j)\cap C^\alpha(\overline{\Omega_j})$ such that $g_j=l_{X_j,b_j}$ on $\partial\Omega_j$ and satisfies  for $\Omega_j^{err}\coloneqq \{x\in\Omega_j : \nabla g_j(x) \notin K\}$:
 \begin{equation} \label{eq:ind1}
     \int_{\Omega_j^{err}} (1+|\nabla g_j|)^p dx < 2^{-(k+1)}  |\Omega_j|,
 \end{equation}
 \begin{equation} \label{eq:ind2}
     \|g_j-l_{X_j,b_j}\|_{C^\alpha(\overline{\Omega_j})} < \delta 2^{-k-2},
 \end{equation}
 \begin{equation} \label{eq:ind4}
     \nabla g_j(x)\in U \quad \text{for a.e. } x\in \Omega_j^{err},
 \end{equation}
\begin{equation} \label{eq:ind3}
     |\{x\in\Omega_j : |\nabla g_j(x)|>t\}| \le M^p(1+|X_j|^p)|\Omega_j|t^{-p} \quad \text{for any } t>0.
 \end{equation}
 
%  \armin{If we have exact reducibility we also have
% \begin{equation} \label{eq:ind3exact}
%      |\{x\in\Omega_j : |\nabla g_j(x)|>t\}| \geq m^p |\Omega_j|t^{-p} \quad \text{for any } t \in (1,\infty),
%  \end{equation}
%  }
 Define $w_{k+1} : \Omega\to\R^d$ by
 \[
 w_{k+1}(x) = \begin{cases}
     g_j(x) & \text{if } x\in\Omega_j; \\
     w_k(x) & \text{if } x\in (\Omega\setminus\Omega_{k,err})\cup \mathcal{N}.
 \end{cases}
 \]
 By the gluing argument again, we obtain a piecewise affine map $w_{k+1}\in W^{1,1}(\Omega)\cap C^\alpha(\overline{\Omega})$ such that $w_{k+1}=l_{X_0,b}$ on $\partial\Omega$ and
 \begin{equation} \label{eq:inderror}
     \Omega_{k+1,err} \coloneqq  \bigcup_j \Omega_j^{err} \cup \mathcal{N}.
 \end{equation}
 From \eqref{eq:ind1}, we have
 \[
 \int_{\Omega_{k+1,err}} (1+|\nabla w_{k+1}|)^p dx = \sum_j \int_{\Omega_j^{err}} (1+|\nabla g_j|)^p dx <\sum_j 2^{-(k+1)}|\Omega_j| \le 2^{-(k+1)}|\Omega|,
 \]
which is \eqref{eq2.17} for $w_{k+1}$.

By the gluing argument and \eqref{eq:ind2}, we obtain
\begin{equation} \label{csseq}
    \|w_{k+1}-w_k\|_{C^\alpha(\overline{\Omega})}\le 2\sup_j \|g_j-l_{X_j,b_j}\|_{C^\alpha(\overline{\Omega_j})} < \delta 2^{-k-1}.
\end{equation}
Thus, it follows from \eqref{eq2.18} that
\[
\|w_{k+1}-l_{X_0,b}\|_{C^\alpha(\overline{\Omega})} \le \|w_{k+1}-w_k\|_{C^\alpha(\overline{\Omega})} + \|w_k-l_{X_0,b}\|_{C^\alpha(\overline{\Omega})} <\delta 2^{-k-1}+\delta (1-2^{-k}) = \delta (1-2^{-(k+1)}),
\]
which is \eqref{eq2.18} for $w_{k+1}$.

Clearly we have 
\[
\begin{split}
|\{x\in\Omega : |\nabla w_{k+1}(x)|>t\}| =& |\{x\in \Omega\setminus\Omega_{k,err} : |\nabla w_k(x)|>t\}| + \sum_j |\{x\in\Omega_j : |\nabla g_j(x)|>t\}| \\
\leq& |\{x\in \Omega : |\nabla w_k(x)|>t\}| + \sum_j |\{x\in\Omega_j : |\nabla g_j(x)|>t\}|.
\end{split}
\]
Thus, from \eqref{eq:ind3}, \eqref{eq2.19} and \eqref{eq2.17} (for $w_k$), we obtain (recall that $X_j = \nabla w_k$ on $\Omega_j$)
\begin{equation*}
    \begin{split}
        |\{x\in\Omega : |\nabla w_{k+1}(x)|>t\}| 
        % &= |\{x\in (\Omega\setminus\Omega_{k,err}) : |\nabla w_k(x)|>t\}| + \sum_j |\{x\in\Omega_j : |\nabla g_j(x)|>t\}| \\
        &\le M^p(1+|X_0|^p)|\Omega|t^{-p} \sum_{i=0}^{k-1} 2^{-i} + \sum_j M^p(1+|X_j|^p)|\Omega_j|t^{-p} \\
        &= M^p(1+|X_0|^p)|\Omega|t^{-p} \sum_{i=0}^{k-1} 2^{-i} + \sum_j M^pt^{-p} \int_{\Omega_j} (1+|\nabla w_k|^p) dx \\
        &= M^p(1+|X_0|^p)|\Omega|t^{-p} \sum_{i=0}^{k-1} 2^{-i} + M^pt^{-p} \int_{\Omega_{k,err}} (1+|\nabla w_k|)^p dx \\
        &\le M^p(1+|X_0|^p)|\Omega|t^{-p} \sum_{i=0}^{k-1} 2^{-i} + M^p t^{-p}\, |\Omega| 2^{-k} \\
        &\leq M^p(1+|X_0|^p)|\Omega| t^{-p} \sum_{i=0}^k 2^{-i},
    \end{split}
\end{equation*}
which is \eqref{eq2.19} for $w_{k+1}$.

If we have an exact reduction, we have \eqref{eq2.19exact} for $w_k$. Observe that $\Omega_{k+1,err} \subset \Omega_{k,err}$ (up to a null set) and $\nabla w_k = \nabla w_{k+1}$ on $\Omega \setminus \Omega_{k+1,err}$ hence

\[
|\{x\in\Omega \setminus \Omega_{err,k+1} : |\nabla w_{k+1}(x)|>t\}| \geq 
|\{x\in\Omega \setminus \Omega_{err,k} : |\nabla w_{k}(x)|>t\}| \geq \frac{1}{100} m^p t^{-p}
\]
for any $t \geq 1$. This gives \eqref{eq2.19exact} for $w_{k+1}$.

%  \iarmin{until here}
% \ToDo As for the exact reducibility, we observe with \eqref{eq2.17} and \eqref{eq:ind3exact} for any $t \in (1,\infty)$, if we choose $\lambda < \frac{1}{2}\min\{m^p,1\}$ and Chebyshev inequality,
% \[
% \begin{split}
% |\{x\in\Omega\setminus\Omega_{err,k+1} : |\nabla w_{k+1}(x)|>t\}| 
% \ge |\{x\in \Omega\setminus\Omega_{err,k} : |\nabla w_k(x)|>t\}| 
% \ge \frac{1}{100}m^p|\Omega|t^{-p}
% %+ \sum_j |\{x\in\Omega_j : |\nabla g_j(x)|>t\}| 
% %-|\{x\in \Omega_{err,k} : |\nabla w_k(x)|>t\}|\\
% %\geq& m^p |\Omega|t^{-p} \ToDo \sum_{i=0}^{k-1} 2^{-i} + \sum_j m^p |\Omega_j|t^{-p} -t^{-p} \int_{\Omega_{err;k}} |\nabla w_k|^p\\
% %\geq& m^p|\Omega|t^{-p}\ToDo \sum_{i=0}^{k-1} 2^{-i} + m^p |\Omega_{err,k}|t^{-p} -t^{-p}|\Omega|\lambda2^{-k} \\
% %\geq & t^{-p} |\Omega| \brac{m^p\ToDo \sum_{i=0}^{k-1} 2^{-i} + m^p  - m^p \lambda 2^{-k} -\lambda 2^{-k} }\\
% \end{split}
% \]
% \iarmin{something is fishy this goes to infinity?}
% \ToDo which is \eqref{eq2.19exact} for $w_{k+1}$.

The condition \eqref{eq:dust} for $w_{k+1}$ follows from \eqref{eq:ind4}, \eqref{eq:inderror}, and \eqref{eq:dust} for $w_k$.
\end{proof}
\begin{proof}[Proof of \Cref{th:thm4.1}]

We have a sequence of maps $(w_k)$ as in \Cref{la:exactPART}. Define $w:\Omega\to\R^d$ by 
\[
w(x) = \lim_{k\to\infty} w_k(x).
\]
Note that since $\Omega_{k+1,err}\subset \Omega_{k,err}$ for any $k$ and $|\Omega_{k,err}|\to 0$ as $k\to\infty$ due to \eqref{eq2.17}, we know that this limit exists for almost every $x\in\Omega$. Moreover,
$\Omega_{err} \coloneqq  \bigcap_k \Omega_{k,err}$ satisfies 
$|\Omega_{err}|=\lim_{k\to\infty} |\Omega_{k,err}|=0$. Since $w=w_k$ on $\Omega \setminus \Omega_{k,err}$, $w$ is a piecewise affine map with $w=l_{X_0,b}$ on $\partial\Omega$. We also have $\nabla w(x)\in K$ for a.e. $x\in\Omega$. Lastly, \eqref{csseq} implies that $(w_k)$ is a Cauchy sequence in $C^\alpha(\overline{\Omega})$, so $(w_k)$ converges to $w$ in $C^\alpha(\overline{\Omega})$. It follows that
\[
\|w-l_{X_0,b}\|_{C^\alpha(\overline{\Omega})} = \lim_{k\to\infty} \|w_k-l_{X_0,b}\|_{C^\alpha(\overline{\Omega})} \leq \delta.
\]
From \eqref{eq2.19}, we have for each $k$,
\[
|\{x\in\Omega : |\nabla w_k(x)|>t\}| \le 2M^p(1+|X_0|^p)|\Omega|t^{-p}.
\]
In particular $\nabla w_k \in L^{(p,\infty)}(\Omega)$ (the weak Lorentz space)  with a uniform bound, thus up to a subsequence $w_k$ weakly converges to $w$ in $W^{1,q}(\Omega)$ for any $q < p$.

Observe that for all $t > 0$ and a.e. $x \in \Omega$
\[
\chi_{\{x\in\Omega : |\nabla w(x)|>t\}} \leq \liminf_{k \to \infty } \chi_{\{x\in\Omega : |\nabla w_k(x)|>t\}}. 
\]
By Fatou's lemma we find 
\[
   |\{x\in\Omega : |\nabla w(x)|>t\}| \leq \liminf_{k \to \infty}    |\{x\in\Omega : |\nabla w_k(x)|>t\}| \leq 2M^p(1+|X_0|^p)|\Omega|t^{-p}
\]

% on $\Omega$ and $|\chi_{\{x\in\Omega : |\nabla w_k(x)|>t\}}| \le 1$ for any $k$, which $1\in L^1(\Omega)$. By the Dominated Convergence theorem, 
% \[
% \begin{split}
%     |\{x\in\Omega : |\nabla w(x)|>t\}| = \int_\Omega \chi_{\{x\in\Omega : |\nabla w(x)|>t\}} dx &= \lim_{k\to\infty} \int_\Omega \chi_{\{x\in\Omega : |\nabla w_k(x)|>t\}} dx \\
%     &= \lim_{k\to\infty} |\{x\in\Omega : |\nabla w_k(x)|>t\}| \\
%     &\le 2M^p(1+|X_0|^p)|\Omega|t^{-p}
% \end{split}
% \]
for any $t>0$. Thus, we established \eqref{eq:wewant4.2c}.

% This implies $\nabla w \in L^{(p,\infty)}(\Omega)$.
% % Hence, for $q<p$,
% % \[
% % \begin{split}
% %     \int_\Omega |\nabla w|^q dx &= \int_0^\infty qt^{q-1} |\{x\in\Omega : |\nabla w(x)|>t\}| dt \\
% %     &= \int_0^1 qt^{q-1} |\{x\in\Omega : |\nabla w(x)|>t\}| dt + \int_1^\infty qt^{q-1} |\{x\in\Omega : |\nabla w(x)|>t\}| dt \\
% %     &\le \int_0^1 qt^{q-1} |\Omega| dt + 2qM^p(1+|X_0|^p)|\Omega| \int_1^\infty t^{q-p-1} dt \\
% %     &=|\Omega| + 2\frac{q}{q-p}M^p(1+|X_0|^p)|\Omega| <\infty.
% % \end{split}
% % \]
% % This shows that $w\in W^{1,q}(\Omega)$ for any $q<p$. 

% In particular, $w\in W^{1,1}(\Omega)$. This concludes the desired results.

If the reduction is \emph{exact}, we observe that for any $k \in \N$, by \eqref{eq2.19exact}
\[
\bigl|\{x\in\Omega\colon \ |\nabla w(x)|>t\}\bigr| \geq
\bigl|\{x\in\Omega\setminus \Omega_{k,err}\colon \ |\nabla w_k(x)|>t\}\bigr| \geq \frac{1}{100} m^p |\Omega|t^{-p}  \quad \text{for any } t \in (1,\infty).  
\]
This gives \eqref{eq:wewant4.2c2} and we can conclude.
\end{proof}

\section{The staircase laminate criterion --- Proof of Theorem~\ref{th:thm4.3}}\label{s:laminatecrit}
We begin with the basic definitions. The only difference to before is the notion of ``construction steps'' being confined to the set $U$.
\begin{definition}[Elementary splitting and elementary splitting with construction steps]
Given probability measures $\nu,\mu\in\mathcal{P}(\R^{d\times m})$ we say that $\mu$ is obtained from $\nu$ \emph{by elementary splitting} if $\nu$ has the form $\nu=\lambda\delta_X+(1-\lambda)\tilde\nu$ for some $\tilde\nu\in\mathcal{P}(\R^{d\times m})$, there exist matrices $B,B'\in\R^{d\times m}$ and $\lambda'\in(0,1)$ such that $X=\lambda' B+(1-\lambda')B'$ and $\rank(B'-B)=1$, and moreover,
\[\mu=\lambda\big(\lambda'\delta_B+(1-\lambda'\big)\delta_{B'})+(1-\lambda)\tilde\nu.\]
Let $U\subset\R^{d\times m}$. If $X\in U$ we say that this elementary splitting of $\nu$ \emph{originates from $U$} or \emph{has construction step in $U$}. 
\end{definition}

\begin{definition}[Laminate of finite order]
A probability measure $\nu\in\mathcal{P}(\R^{d\times m})$ is called \emph{a laminate of finite order} if it can be obtained by starting from a Dirac mass and applying elementary splitting a finite number of times. The set of all laminates of finite order on $\R^{d\times m}$ is denoted by $\mathcal{L}(\R^{d\times m})$.
\end{definition}

\begin{definition}[Laminate of finite order with construction steps]\label{def:laminatewithconstructionsteps}
Let $\nu\in \mathcal{L}(\R^{d\times m})$ and $U\subset \R^{d\times m}$. We call $\nu$ a \emph{laminate of finite order $\nu\in\mathcal{L}(\R^{d\times m})$ with construction steps in $U$} if every elementary splitting originates from an element in $U$.
\end{definition}

%\iarmin{Precise definition} Laminate of finite order {with construction steps in $U$} cf. \cite[Section 2.1]{KMSX24} {meaning that every elementary splitting originates from an element in $U$}\Armin{this does \emph{not} mean that $\supp \nu \in U$!!!}.
%\end{definition}
Observe that construction steps in $U$ does not mean the support of $\nu$ belongs to $U$. Indeed, for a laminate of finite order $\nu$, even with construction steps in $U$, in general $\supp\nu \not \subset U$. However, the baricenter of $\nu$ must belong to $U$.

\begin{definition}[Barycenter]
Let $\nu\in\mathcal{L}(\R^{d\times m})$ be of the form $\nu=\sum_{i=1}^N\lambda_i\delta_{X_i}$. The \emph{barycenter} of $\nu$ is defined as the sum $\sum_{i=1}^N\lambda_i X_i$.
%and will be denoted as $\bar\nu$.
\end{definition}

\begin{definition}[Staircase laminate, {\cite[Proposition 3.1]{KMSX24}}]\label{defstair}
Let $K\subset\R^{d\times m}$ and $X_0\notin K$. Suppose that there exists a sequence of pairwise different matrices $X_n\in\R^{d\times m}\setminus K$, $n=0,1,2,\ldots$ called \emph{soaring sequence}, a sequence of probability measures $\mu_n\in\mathcal{P}(K)$ supported in $K$ such that each of $\mu_n$ is a finite sum of Dirac masses, and scalars $\gamma_n\in (0,1)$ such that
\begin{enumerate}
    \item for each $n\in\N$, the probability measure (finite sum of Dirac masses)
    \[\omega_n\coloneqq(1-\gamma_n)\mu_n+\gamma_n\delta_{X_n}\]
    is a laminate of finite order with barycenter $X_{n-1}$;
    %$\overline{\omega_n}=X_{n-1}$ if we have this indication in the definition
    \item the sequence $|X_n|$ is strictly monotone increasing with $\lim_{n\to\infty}|X_n|=\infty$;
    \item $\lim_{n\to\infty}\beta_n=0$, where $\beta_n\coloneqq\prod_{k=1}^n \gamma_k$, $\beta_0=1$.
\end{enumerate}
Define the probability measures $\nu^N$, $N=1,2,3,\ldots$, by iteratively replacing $\delta_{X_{n-1}}$ by $\omega_n$ for $1\leq n\leq N$, i.e., by 
\[\nu^N=\sum_{n=1}^N \beta_{n-1}(1-\gamma_n)\mu_n+\beta_N\delta_{X_N}.\]
Then, $\nu^N$ is a laminate of finite order with $\supp\nu^N\subset K\cup\{X_N\}$ and barycenter $X_0$. Moreover, for any Borel set $E\subset\R^{d\times m}$, the limit
\[\lim_{N\to\infty}\nu^N(E)=\nu^\infty(E)\]
exists and defines a probability measure $\nu^\infty$ with $\supp\nu^\infty\subset K$ and barycenter $X_0$. This probability measure $\nu^\infty$ will be called the \emph{staircase laminate}.
%\iarmin{copy the definition from the 4 authors, maybe state more clear that $\mu_n$ is sum of dirac masses. Just replace $A_n$ by $X_n$ and call it the soaring sequence}
\end{definition}
\begin{definition}[Staircase laminate with construction steps in $U$ supported in $K$]\label{def:SLKU}
Let $U,K\subset\R^{d\times m}$.
%, $K\cap U=\emptyset$.
We say that a staircase laminate $\nu\in\mathcal{P}(\R^{d\times m})$ belongs to $SL(K,U)$ if the soaring sequence of $\nu$ is a sequence of laminates of finite order with construction steps in $U$ (in particular $\omega_n$ is a laminate of finite order with construction steps in $U$ starting from $\delta_{X_{n-1}}$), and $\supp \nu \in K$.
\end{definition}

The lemma below is adapted from \cite[Lemma~2.1]{KMSX24}. 
In the original formulation, the conclusion stated in part~\eqref{wiggle:error} is not written explicitly, although it follows directly from the construction given there. 
Since this additional property will play a crucial role later, we include it here and present the argument in full detail.
This lemma serves as a standard building block in convex integration.

Informally, the lemma produces a map $u$ whose gradient takes the values $X_1$ and $X_2$ on most of the domain, with a small exceptional set $\Omega_{\mathrm{err}}$ on which other values may appear. We refer to the set $\nabla w(\Omega_{\mathrm{err}})$ as the \emph{dust}. The key point in part~\eqref{wiggle:error} is that for a suitable open set $U$ we have $\nabla w (x) \in U$ for almost all $x\in \Omega_{err}$, see \eqref{eq:dustinU}. The choice of this neighborhood will be essential in our proof of Theorem~\ref{th:AFS}.

\begin{lemma}[Wiggle lemma]\label{la:basicwiggle}
Let $X, X_1,X_2 \in \R^{d\times m}$ be matrices such that
\[
 \rank(X_1-X_2) = 1 \quad \text{and } X = \lambda X_1 + (1-\lambda) X_2
\]
for some $\lambda \in (0,1)$.

Fix an arbitrary $b \in \R^d$, $\alpha \in (0,1)$, $\eps > 0$, and a regular domain $\Omega \subset \R^m$. There exists a piecewise affine Lipschitz map $w\colon \Omega \to \R^d$ such that

\begin{enumerate}
\item \label{wiggle:boundary}$w(x) = Xx+b$ on $\partial \Omega$;
\item \label{wiggle:maxnorm}$|\nabla w(x)| \leq \max\{|X_1|,|X_2|\}$ for a.e. $x \in \Omega$;
\item \label{wiggle:pmeps1}$(1-\eps) \lambda |\Omega| \leq |\{x \in \Omega\colon \nabla w(x)= X_1\}| \leq (1+\eps) \lambda |\Omega|$;
\item \label{wiggle:pmeps2}$(1-\eps) (1-\lambda) |\Omega| \leq |\{x \in \Omega\colon \nabla w(x)= X_2\}| \leq (1+\eps) (1-\lambda) |\Omega|$.
\item \label{wiggle:error}
Define
\[
\Omega_{\mathrm{err}} \coloneqq \{ x \in \Omega \colon  \nabla w(x) \notin \{X_1, X_2\} \}.
\]
We refer to the image $\nabla w(\Omega_{\mathrm{err}})$ as the \emph{dust}. Then
\[
|\Omega_{\mathrm{err}}| \leq \varepsilon\,|\Omega|
\]
and
\begin{equation}\label{eq:nablaumAll1}
|\nabla w(x) - X| < \varepsilon \qquad \text{for a.e. } x \in \Omega_{\mathrm{err}}.
\end{equation}
In particular, if $U \subset \mathbb{R}^{d\times m}$ is open and $X \in U$, then we can ensure that
\begin{equation}\label{eq:dustinU}
\nabla w(x) \in U \qquad \text{for a.e. } x \in \Omega_{\mathrm{err}}.
\end{equation}
\item \label{eq:wiggle:Calphaclose} $\|f-(X x + b)\|_{C^\alpha(\overline{\Omega})} < \eps$.

%In particular, if for some $\delta, \bar{a},\gamma > 0$ we have $X \in \mathscr{U}_{\bar{a},\delta,\gamma}$ we can ensure that $\nabla u(x) \in \mathscr{U}_{\bar{a},\gamma,\delta}$ for a.e. $x\in\Omega_{err}$.
\end{enumerate}
\end{lemma}

\begin{proof}
Since $X_2 - X_1$ is rank one, we can write for some $\xi \in \R^d \setminus \{0\}$, $\zeta \in \S^{m-1} \subset \R^m$,
\[
 X_1 - X_2 = \xi \otimes \zeta.
\]
We then have
\[
 X_1 -X = (1-\lambda) (X_1 - X_2) = (1-\lambda)\xi \otimes \zeta
\]
and
\[
 X_2 -X = -\lambda \xi \otimes \zeta.
\]
Let $r \in (0,1)$ be a small number to be chosen later. Extend $\zeta$ to obtain a basis $(\zeta,\zeta_2,\ldots,\zeta_m)$ of $\R^m$ such that $|\zeta_i|<r$ for all $i=2,3,\ldots,m$.
%, we define also $\zeta_{n+1}\coloneqq-\zeta_{2}, \ldots,\zeta_{2n-1}\coloneqq-\zeta_{m}$.
Then, the set
\[
\begin{split}
 \Omega_0 \coloneqq
 %& \left \{x \in \R^m: \langle x,\zeta_i\rangle > -1\text{ for $i=2,3,\ldots,2n-1$ and } |\langle x,\zeta\rangle| < 1 \right \}\\
 \left \{x \in \R^m\colon  |\langle x,\zeta_i\rangle| < 1\text{ for $i=2,3,\ldots,m$ and } |\langle x,\zeta\rangle| < 1 \right \}
 \end{split}
\]
is a regular domain, and it is convex with $0 \in \Omega_0$.
Define $f_N\colon\Omega_0\to\R$ as follows
\[
 f_N(x) \coloneqq 
 %\min\left\{ \min_{i\in\{2,3,\ldots,2n-1\}} 1+ \langle x,\zeta_i\rangle, \frac{1}{N} h(N \langle x,\zeta\rangle) \right \},
 \min\left\{ \min_{i\in\{2,3,\ldots,m\}} 1+ \langle x,\zeta_i\rangle, \min_{i\in\{2,3,\ldots,m\}} 1 - \langle x,\zeta_i\rangle, \frac{1}{N} h(N \langle x,\zeta\rangle) \right \},
\]
where $N\in\N$ and $h\colon \R \to [0,\infty)$ is the typical (Lipschitz) sawtooth function extended periodically to $\R$, i.e., $h(x) = h(x-1)$ for all $x \in \R$, $h(0) = h(1) = 0$ and
\begin{equation}\label{eq:valuesofh'}
|\{t \in [0,1]\colon \quad h'(t) = 1-\lambda \}| = \lambda, \quad |\{t \in [0,1]\colon \quad h'(t) = -\lambda \}| = 1-\lambda.
\end{equation}
Then $f_N$ is a piecewise affine Lipschitz function, $f_N=0$ on $\partial\Omega_0$, and for a.e. $x\in\Omega_0$,
\[
\begin{split}
 \nabla f_N(x) &\in 
%  $\{\zeta_2,\zeta_3,\ldots, \zeta_{2n-1},h'(N\langle x,\zeta\rangle)\zeta\}\\
%  &\quad = 
 \{\pm \zeta_2,\pm \zeta_3,\ldots, \pm \zeta_{m},h'(N\langle x,\zeta\rangle)\zeta\}.
\end{split}
 \]
Since $h'\in\{1-\lambda,-\lambda\}$ a.e. and $X_1-X=(1-\lambda)\xi\otimes\zeta$, and $X_2-X=-\lambda\xi\otimes\zeta$ we obtain for a.e. $x\in\Omega_0$
\[
 \xi \otimes \nabla f_N (x) \in \{\pm\xi \otimes \zeta_2,\pm\xi \otimes \zeta_3,\ldots,\pm\xi\otimes\zeta_m, X_1-X,X_2-X\}.
\]
Setting
\[
g_N(x) \coloneqq Xx + b+ \xi f_N(x)
\]
we have for a.e. $x\in\Omega_0$
\begin{equation}\label{nablagerror}
 \nabla g_N(x) \in \{X\pm\xi \otimes \zeta_2,X\pm \xi \otimes \zeta_3,\ldots,X\pm\xi \otimes \zeta_m, X_1,X_2\}.
\end{equation}

We observe that $g_N$ now satisfies the claim of the Lemma if we replace $\Omega$ by $\Omega_0$, except possibly for \eqref{eq:wiggle:Calphaclose}.

Indeed, property \eqref{wiggle:boundary} clearly holds w.r.t to $g_N$ and $\Omega_0$ for any $N$. Moreover, 
\[
 |X| = |\lambda X_1 + (1-\lambda)X_2| < \lambda|X_1| + (1-\lambda)|X_2|\le \max\{|X_1|, |X_2|\} 
\]
and for any $i\in\{2,\ldots,m\}$
\[
 |X\pm \xi\otimes \zeta_i| \le |X|+|\xi||\zeta_i| <|X|+r|\xi|, 
\]
thus choosing $r < \frac{\max\{|X_1|, |X_2|\}-|X|}{|\xi|}$ we obtain property \eqref{wiggle:maxnorm} w.r.t. to $g_N$ and $\Omega_0$.

Now we focus on properties \eqref{wiggle:pmeps1} and \eqref{wiggle:pmeps2}.

It is possible to choose $N=N(r)$ large enough so that 
\begin{equation}\label{eq:hsmallerthanthesum}
\frac{1}{N}h(N \langle x,\zeta\rangle)<\min\left\{\min_{2\leq i\leq m}(1+\langle x,\zeta_i\rangle),\min_{2\leq i\leq m}(1-\langle x,\zeta_i\rangle)\right\}\quad \text{in } (1-r)\Omega_0.
\end{equation}
This follows from the fact that
\[(1-r)\Omega_0\coloneqq\{x\in\R^m\colon 1-r>\langle x,\zeta_i\rangle>r-1\text{ for }i=2,3,\ldots,m\text{ and }|\langle x,\zeta\rangle|<1-r\}.\]

In particular we have 
\[
g_N(x) = Xx+b+\frac{1}{N}h(N \langle x,\zeta\rangle) \text{ in $(1-r)\Omega_0$}.
\]
and thus 
\[
\nabla g_N(x) \in \{X_1,X_2\} \quad \text{in $(1-r)\Omega_0$}
\]
And, taking $N=N(r)$ even larger, we can ensure 
\[
\begin{split}
\abs{\{x \in \Omega_0: \nabla g_N(x) = X_1\}}
&\geq \abs{\{x \in (1-r)\Omega_0: \nabla g_N(x) = X_1\}}\\ 
&= \abs{\{x \in (1-r)\Omega_0: h'(N\langle x,\zeta\rangle) = 1-\lambda\}}\\
&\geq \lambda (1-r)^2 \abs{\Omega_0}
\end{split}
\]
and similarly
\[
\begin{split}
\abs{\{x \in \Omega_0: \nabla g_N(x) = X_2\}}
\geq \abs{\{x \in (1-r)\Omega_0: h'(N\langle x,\zeta\rangle) = -\lambda\}}
\geq (1-\lambda) (1-r)^2 \abs{\Omega_0}
\end{split}
\]
Since $X_1 \neq X_2$ this also implies
\[
\begin{split}
\abs{\{x \in \Omega_0: \nabla g_N(x) = X_1\}}
&\leq |\Omega_0| - \abs{\{x \in \Omega_0: \nabla g_N(x) = X_2\}}\\ 
&\leq \brac{1- (1-\lambda) (1-r)^2} \abs{\Omega_0}\\
&= \lambda \frac{\brac{1- (1-\lambda) (1-r)^2}}{\lambda} \abs{\Omega_0}\\
\end{split}
\]
and
\[
\begin{split}
\abs{\{x \in \Omega_0: \nabla g_N(x) = X_2\}}
\leq |\Omega_0| - \abs{\{x \in \Omega_0: \nabla g_N(x) = X_1\}} 
% \leq& \brac{1- \lambda (1-r)^2} \abs{\Omega_0}\\
\leq (1-\lambda)\frac{\brac{1- \lambda (1-r)^2}}{1-\lambda} \abs{\Omega_0}
\end{split}
\]
So if we take $r$ so small (and $N=N(r)$ large enough accordingly) so that 
\[
\frac{\brac{1- (1-\lambda) (1-r)^2}}{\lambda}, \frac{\brac{1- \lambda (1-r)^2}}{1-\lambda} \leq 1+\eps
\]
and 
\[
(1-r)^2 \geq (1-\eps),
\]
then properties \eqref{wiggle:pmeps1} and \eqref{wiggle:pmeps2} are satisfied w.r.t $\Omega_0$, and $g_N$.

Next we establish property \eqref{wiggle:error}, so we examine what happens on the set 
\[
 \Omega_{0,err} \coloneqq \left \{x \in \Omega_0\colon  \nabla g_N(x) \not \in \{X_1,X_2\} \right \}.
\]
From properties \eqref{wiggle:pmeps1} and \eqref{wiggle:pmeps2} it is clear that $|\Omega_{0,err}| \leq \eps |\Omega_0|$.  
% By \eqref{wiggle:pmeps1} and \eqref{wiggle:pmeps2} we get that $|\Omega_{0,err}|< \eps |\Omega_0|$. 
Moreover, from \eqref{nablagerror} we deduce that 
\[
 \nabla g_N(x) \in \left \{X\pm\xi \otimes \zeta_2,X\pm\xi \otimes \zeta_3,\ldots,X\pm\xi \otimes \zeta_m \right \} \quad \text{for a.e. $x\in\Omega_{0,err}$}.
\]
Thus,
\[
 |\nabla g_N(x) - X| \le \max_{2\le i \le m}|\xi||\zeta_i| \leq r|\xi| \quad \text{for a.e. $x\in\Omega_{0,err}$}.
\]
So taking $r$ possibly even smaller, ensuring that $r < \frac{\eps}{|\xi|}$, we obtain  \eqref{eq:nablaumAll1}. This ensures property \eqref{wiggle:error} for $g_N$ on $\Omega_0$.

This shows that $g_N$ satisfies the conclusion of the lemma in $\Omega_0$, except possibly for  \eqref{eq:wiggle:Calphaclose}.
  
% Lastly, observe that if $X\in U$ 
% %$\mathscr{U}_{\bar{a},\delta,\gamma}$ is an open set for $\bar{a},\delta,\gamma > 0$. It follows that if $X \in \mathscr{U}_{\bar{a},\delta,\gamma}$ 
% we can again choose $r>0$ small enough such that 
% %$X \in \mathscr{U}_{\bar{a},\delta,\gamma}$ implies
% \eqref{eq:nablaumAll1} implies
% %\[
% % \nabla g_N \in \mathscr{U}_{\bar{a},\delta,\gamma}\quad \text{a.e. in $\Omega_{0,err}$}.
% %\]
% \[
%  \nabla g_N \in U\quad \text{a.e. in $\Omega_{0,err}$}.
% \]
In order to find $w$ on $\Omega$ (including \eqref{eq:wiggle:Calphaclose}), we cover $\Omega$ by countably many tiny rescaled copies of $\Omega_0$ and set $w$ to be a suitably rescaled version of $f_N$ --- this is the argument from the beginning of the proof of \Cref{la:exactPART}, see also the \emph{rescaling and covering argument} in \cite[Section 2.1]{KMSX24}.
\end{proof}
Iterating \Cref{la:basicwiggle} we obtain.

\begin{lemma}\label{la:wiggleiter}
Let $U \subset \R^{d\times m}$ be open.

Assume that $\nu$ is a laminate of finite order with construction steps in  $U$ (see \Cref{def:laminatewithconstructionsteps}) and with barycenter $X_0 \in U$,
%which is constructed similarly to \cite[Section 2.1]{KMSX24}, i.e.,
\[
 \nu = \sum_{j=1}^J \lambda_j \delta_{X_j} \text{ where }\lambda_j>0, \sum_{j=1}^J\lambda_j=1, \text{ and } X_k\neq X_j\text{ for }k\neq j.
 \]
%meaning that every elementary splitting originates from an element in $U$.\footnote{This does not mean that every $X_j\in U$.}
%\Armin{this does \emph{not} mean that $X_j \in \mathscr{U}_{\bar{a},\delta,\gamma}$!!!}.

Then, given any $b\in\R^d$, $\eps>0$, $\alpha \in (0,1)$ and any regular domain $\Omega\subset\R^m$, there exists a piecewise affine Lipschitz map $w\colon\Omega\to\R^d$ such that

\begin{enumerate} 
\item $w(x)=X_0 x+b$ for $x\in\partial \Omega$;
\item \label{max} $\|\nabla w\|_{L^\infty} \leq \max_{1\le j\le J} |X_j|$;
\item \label{epsj} we have
\[
 (1-\eps) \lambda_j |\Omega| \leq |\{x \in \Omega\colon  \nabla w(x) = X_j\}| \leq (1+\eps) \lambda_j |\Omega| \quad \text{ for } j=1,\ldots,J;
\]
% \item $\nabla u \in U$ a.e. in $\Omega_{err}$ where
% \[
%  \Omega_{err} \coloneqq \Omega \setminus \bigcup_{j=1}^J \{x\colon  \nabla u(x) = X_j\}.
% \]
\item\label{it:afterproof} Define
\[
\Omega_{\mathrm{err}} \coloneqq \Omega \setminus \bigcup_{j=1}^J \{x\colon  \nabla w(x) = X_j\}.
\]
Then,
\begin{equation}\label{afterproof}
|\Omega_{\mathrm{err}}| < \varepsilon\,|\Omega|
\end{equation}
and the dust is in $U$, i.e., $\nabla w \in U$ a.e.\ in $\Omega_{err}$.
\item\label{it:Calphaclose}$\|w-(X_0x + b)\|_{C^\alpha(\Omega)} \leq \eps$. 
\end{enumerate}
\end{lemma}
% Observe that from \eqref{epsj} it follows that 
% \begin{equation}\label{afterproof}
%     |\{x\in\Omega\colon\nabla u(x)\notin\supp\nu\}\leq\eps|\Omega|
% \end{equation}
% since $\sum_{j=1}^J\lambda_j=1$.

% \begin{proof}
% \ikasia{This proof should be written with more details}

% Since $\nu$ is a laminate of finite order, this follows from an iterative usage of \Cref{la:basicwiggle} and from the definition of a piecewise affine mapping (i.e., we can do elementary splittings and still work in regular domains).
% \end{proof}
Our next goal is the following version of \cite[Proposition 4.4]{KMSX24}. Recall the definition of $SL(K,U)$ from \Cref{def:SLKU}.
\begin{proposition}\label{pr:onestaircase}
Let $K \subset \R^{d\times m}$ be any set and $U \subset \R^{d\times m}$ be open.
%such that $K\cap U = \emptyset$. 
Let $\alpha \in (0,1)$, $q \in (1,\infty)$ be arbitrary.

Assume that $\nu^\infty \in SL(K,U)$, i.e., is a staircase laminate supported on a set $K$ with construction steps in $U$, i.e., the \emph{soaring sequence} $X_n \in U$ and barycenter $X_0\in U$.

Assume moreover that for some $p > 1$ we have for some $M > 0$ the upper bound estimate
\begin{equation}\label{asslam}
 \nu^\infty (\{X\colon  |X|>t\}) \leq M\, t^{-p} \quad \forall t >0.
\end{equation}
Then for each $b\in\R^d$, $\eps > 0$ and each regular domain $\Omega\subset\R^m$ there exists a piecewise affine map $w\colon\Omega\to\R^d$, $w\in W^{1,1}(\Omega,\R^d) \cap C^{\alpha}(\bar{\Omega},\R^d)$ such that
\begin{enumerate}
\item \label{boundary} $w(x)=X_0 x+b$ for $x\in\partial \Omega$;
\item \label{error:int} there exists a set $\Omega_{err} {\supset} \{x \in \Omega\colon \nabla w(x) \not \in K\}$ for which we have
 \[
  \int_{\Omega_{err}} (1+|\nabla w|)^q \leq \eps |\Omega|;
 \]
 \item \label{borelerr} for each Borel set $E \subset \R^{d\times m}$
\begin{equation}\label{eq:upperlowernuinfty}
 (1-\eps) \nu^\infty(E) \leq \frac{|\{x \in \Omega \setminus \Omega_{err}\colon \nabla w(x) \in E\}|}{|\Omega|} \leq (1+\eps) \nu^\infty(E);
\end{equation}
%\iarmin{I think \cite[(4.25c)]{KMSX24} is false, $\frac{|\{x \in \Omega: \nabla u(x) \in E\}|}{|\Omega|}$ can be nonzero if $E \cap \supp \nu^N = \emptyset$}\Igor{I agree with this.}
\item \label{error} the dust is in $U$, i.e., $\nabla w \in U$ a.e. in $\Omega_{err}$.
\end{enumerate}

\end{proposition}

\begin{proof}
We follow the argument of \cite[Proposition 4.4]{KMSX24}, adapting it to our present setting for completeness.

Recall from the \Cref{defstair} that $\nu^\infty\in SL(K,U)$ is a limit of a sequence of laminates of finite order $\nu^N$ with construction steps in $U$ (in particular, $X_n\in U$), barycenter $X_0$
\begin{equation}\label{eq:lfoinSL}
 \nu^N = \sum_{n=1}^N\beta_{n-1}(1-\gamma_n)\mu_n + \beta_N \delta_{X_N}, \quad \beta_n\coloneqq \prod_{k=1}^n\gamma_k,\  \beta_0 =1,
 \end{equation}
 and $\supp \mu_n \in K$.
 
Recall also that
 \begin{equation}\label{eq:omegaN}
\omega_N = (1-\gamma_N)\mu_N + \gamma_N\delta_{X_N} \quad \text{ has barycenter } X_{N-1}
\end{equation}
and is a laminate of finite order starting from $\delta_{X_{N-1}}$ with construction steps in $U$.

Let us write
\begin{equation}\label{nuN-easierform}
\begin{split}
\omega_N&=\sum_{i=0}^{J_N}\theta_{N,i}\delta_{Y_{N,i}},\text{ where } (1-\gamma_N)\mu_N=\sum_{i=1}^{J_N}\theta_{N,i}\delta_{Y_{N,i}} \text{ and } \theta_{N,0} = \gamma_N,\, Y_{N,0} = X_N;\\
|Y_N|&\coloneqq \max_{0\le i\le J_N}{|Y_{N,i}|} = \max\{|X_N|,\max_{1\le i\le J_N}|Y_{N,i}|\}.
\end{split}
\end{equation} 
%\iigor{More characters but it lacked some details I reckon.}

Let $q\in(1,\infty)$, $\eta>0$, and set 
\[
 c_N = \prod_{j=1}^N (1+2^{-j}\eta).
\]
Similarly to \cite[Proposition 4.4]{KMSX24} we will construct a sequence of piecewise affine Lipschitz maps $w_N\colon\Omega\to\R^d$ satisfying the following properties:

Set
\[\Omega^N_{ind} \coloneqq \{x \in \mathring\Omega\colon  \nabla w_N(x) = X_N\},\]
where $\mathring\Omega \coloneqq  \Omega \setminus \mathcal{N}$ is the decomposition defined in \Cref{piecewiseaf} corresponding to the piecewise affine map $w_N$.
\begin{enumerate}[label=(\roman*)]
 \item \label{boundaryI} $w_N = X_0 x+b$ on $\partial \Omega$;
 \item \label{errorN} $\nabla w_N \in U$ a.e. on some set $\Omega^N_{err} {\supset} \{x \in \Omega: \nabla w_N(x) \not \in \supp \nu_N\}$ with $\Omega_{err}^N \cap \Omega^N_{ind} = \emptyset$;
 \item \label{errorN:int} $\int _{\Omega_{err}^N} (1+|\nabla w_N|)^q \leq \eta |\Omega| (1-2^{-N})$;
 \item \label{outside:ind} $w_N = w_k$ in $\Omega \setminus \Omega^k_{ind}$ for all $1 \leq k \leq N-1$; $\Omega^N_{err} \supset \Omega^{N-1}_{err}$;
 \item \label{borelerrN} $c_N^{-1}|\Omega| \nu^N(E) \leq |\{x \in \Omega \setminus \Omega_{err}^{N}\colon  \nabla w_N(x) \in E\}| \leq c_N|\Omega| \nu^N(E)$ for any Borel set $E\subseteq\R^{d\times m}$;
 \item\label{dustdoesntseeXi} $\nabla w_N(x) \not \in \bigcup_{i =N}^\infty \{X_i\}$ for a.e. $x \in \Omega_{err}^N$;
 \item \label{Calphaconv} 
\[ \|w_N - w_{N-1}\|_{C^{\alpha}(\bar\Omega,\R^d)} < 2^{1-N}\eta;
\]
\item\label{OmegaIndest} 
\[
c_N^{-1} \beta_N |\Omega| \leq |\Omega_{ind}^N| \leq c_N \beta_N |\Omega|.
\]
\end{enumerate}

For $N=1$ 
% \begin{equation}\label{nu1}
% \nu^1=\alpha_1\mu_1+\beta_1\delta_{X_1} = \sum_{i=0}^k\theta_{i,1}\delta_{Y_{i,1}},\text{ where } \alpha_1\mu_1=\sum_{i=1}^k\theta_{i,1}\delta_{Y_{i,1}} \text{ and } \theta_{0,1} = \beta_1,\, Y_{0,1} = X_1.  
% \end{equation}
% and let $|Y|\coloneqq \max_{0\le i\le k}{|Y_{i,1}|} = \max\{|X_1|,\max_{1\le i\le k}|Y_{i,1}|\}$. 
we find a Lipschitz $w_1$ with required properties as a consequence of \Cref{la:wiggleiter} with $\nu^1=\omega_1$ starting from $\delta_{X_0}$ (with construction steps in $U$). Indeed, properties \ref{boundaryI}, \ref{errorN} are straightforward for $\Omega_{err}^1 \coloneqq  \{x \in \Omega: \nabla w_n(x) \not \in \supp \nu_1\}$. %Point \ref{errorN} follows also from \Cref{la:wiggleiter} since the set $\Omega_{err}^1=\{x\in\Omega\colon\nabla u_1(x)\notin\supp \nu^1\}$ for a laminate $\nu^1$ is constructed the same as the set $\Omega_{err}$ in \Cref{la:wiggleiter} for a laminate $\nu$ therein.
Using the notation in \eqref{nuN-easierform}, the property \ref{errorN:int} follows from items \eqref{max} and \eqref{it:afterproof} in \Cref{la:wiggleiter}
\[
\begin{split}
  \int_{\Omega_{err}^1} (1+|\nabla w_1|)^q &\leq |\Omega_{err}^1| \|(1+|\nabla w_1|)^q\|_{L^\infty} \leq \tilde \eps|\Omega|(1+|Y_1|)^q \leq \eta|\Omega|(1-2^{-1}),
  \end{split}
\]
where $\tilde \eps=\tilde \eps(q,\eta,|Y_1|)>0$ in \Cref{la:wiggleiter} is chosen to be small enough.
%, where $\bigcup_i Y_{i,1}$ and $X_1$ are the matrices on which $\nu^1$ is supported.
%, i.e., $\nu^1=\alpha_1\mu_1+\beta_1\delta_{X_1}$ for $\mu_1=\sum_i\theta_{i,1}\delta_{Y_{i,1}}$ 
%\Kasia{But from how it is written it's not clear that $\mu_1 = \delta_{Y_1}$.}\Igor{It was a mistake. Now it should be correct}
%\iigor{I added this argument since I do not like the fact that $\eps$ chosen in Lemma 5.2 needed to depend on $u_1$ (on the $L^\infty$ norm, precisely). For me it should be like that: first choose some $\eps$ and obtain $u_1$ with given properties from Lemma 5.2. Then observe that $\|u_1\|_{L^\infty}$ is controlled from above by the maximum of the norms of, fixed at the beginning!, matrices in support of the laminate. Then having this bound take a new $\eps$ and obtain (possibly new) $u_1$ with better properties in terms of the integral inequality.}

Property \ref{outside:ind} is empty for $N=1$.

Property \ref{dustdoesntseeXi} can be satisfied by choosing $\tilde{\eps}$ suitably small, then for a.e. $x \in \Omega$ the gradient $\nabla w_n(x)$ is either close to $X_0$ or $\nabla w_n \in \supp \omega_1$. In $\Omega^1_{err}$ we thus have $\nabla w_n(x)$ close to $X_0$ ($U$ is open). Since $\lim_{n \to \infty} |X_n| = \infty$ and using the fact that $X_n$ are pairwise distinct, we can ensure that $\nabla w_N(x) \neq X_i$ for any $i \neq 0$.

To verify property \ref{borelerrN} for $N=1$ we first note that if $E\cap \supp \nu^1 = E\cap\bigcup_{i=0}^{J_1}  \{Y_{1,i}\}=\emptyset$, then $\nu^1(E)=0$ and $|\{x \in \Omega \setminus \Omega_{err}^{1}\colon  \nabla w_1(x) \in E\}|=0$. Now, let $E\in \R^{d\times m}$ be any Borel set such that $E\cap\bigcup_{i=0}^{J_1}  \{Y_{1,i}\} \neq \emptyset $, then 
\[
E\cap\bigcup_{i=0}^{J_1}  \{Y_{1,i}\}=\bigcup_{i\in I_1} \{Y_{1,i}\} \quad \text{ for some } I_1 \subset \{0,1,\ldots,J_1\}.
\]
Then,
\[
|\{x\in\Omega\setminus\Omega_{err}^1\colon\nabla w_1(x)\in E\}|=\sum_{i\in I_1}|\{x\in\Omega\colon\nabla w_1(x)=Y_{1,i}\}|.
\]
Thus, from \Cref{la:wiggleiter} item \eqref{epsj} we obtain
\[
|\Omega|(1-\tilde{\eps})\Big(\sum_{i\in I_1} \theta_{1,i}\Big)\leq|\{x\in\Omega\setminus\Omega_{err}^1\colon\nabla w_1(x)\in E\}|\leq|\Omega|(1+\tilde{\eps})\Big(\sum_{i\in I_1} \theta_{1,i}\Big).
\]
Noting that $\nu^1(E)=\sum_{i\in I_1} \theta_{1,i}$ and choosing $\tilde{\eps}\leq \frac{\eta}{2+\eta}$, we obtain 
\[c_1^{-1}|\Omega|\nu^1(E)\leq|\{x\in\Omega\setminus\Omega_{err}^1\colon\nabla w_1(x)\in E\}|\leq c_1|\Omega|\nu^1(E).\]
This gives \ref{borelerrN} for $N=1$. Similarly we have \ref{OmegaIndest}, and this concludes the basis step of the inductive construction.

Given \underline{$w_N$ the next function $w_{N+1}$} is constructed by replacing $w_{N}$ on $\Omega^N_{ind}$ according to the next laminate of finite order decomposition. 
%That is to obtain $u_{N+1}$ we will apply \Cref{la:wiggleiter} with $\nu^{N+1}$ defined in \eqref{eq:lfoinSL} on connected components of $\Omega^N_{ind}$. 

By construction, there exists a decomposition of $\Omega_{ind}^N$ into a disjoint union of countably many regular domains: 
\[
\Omega_{ind}^N=\bigcup_{i=1}^{L_N} \Omega_i^N, \quad L_N \in \N\cup\{\infty\},  
\]
such that $w_N$ is affine on $\Omega_i^N$ with $\nabla w_N \equiv X_N$ in $\Omega_i^N$, i.e., $w_N(x)=X_Nx+b_i$ on each $\Omega_i^N$ for some $b_i \in \R$. On each $\Omega_i^N$ we apply \Cref{la:wiggleiter} with $\omega_{N+1}$ ($\tilde{\eps}_i$ will be chosen later) obtaining this way $w_i^{N+1}$. Set
\[
\Omega^{N+1}_{err,i} \coloneqq  \{x \in \Omega_i^N: \nabla w_i^{N+1} \not \in \supp \omega_N\} \supset \{x \in \Omega_i^N: \nabla w_i^{N+1} \not \in \supp \nu_N\}.
\]
As above for the case $N=1$, we find that for any Borel set $E \subset \R^{d \times m}$  
\begin{equation}\label{eq:upperlowernuinftyOmegaiN}
 \frac{c_N}{c_{N+1}} \omega_{N+1}(E) \leq \frac{|\{x \in \Omega_{i}^N \setminus \Omega_{err,i}^{N+1}\colon \nabla w_i(x) \in E\}|}{|\Omega_i^N|} \leq \frac{c_{N+1}}{c_N} \omega_{N+1}(E);
\end{equation}

We define
\begin{equation}\label{eq:definitionuN+1}
 w_{N+1} \coloneqq \left\{ \begin{array}{ll}
                            w_N & \text{ outside } \Omega^N_{ind}\\
                            w_i^{N+1} &\text{ in } \Omega_i^N
                           \end{array}
 \right.
\end{equation}
and we set 
\[
\Omega_{err}^{N+1} \coloneqq  \Omega^N_{err} \; \dot{\cup} \; \dot{\bigcup}_{i=1}^{L_N} \, \Omega_{err,i}^{N+1} \; \cup \mathcal{N}.
\]
This immediately gives \ref{boundaryI} and \ref{outside:ind}.

Note that the map $w_N$ is Lipschitz, with 
\[
|\nabla w_N(x)| \le \max\left\{\max_{1\le i\le N}\max_{1\le j\le J_i}|Y_{i,j}|,|X_N|\right\}\quad \text{ for a.e. }x\in\Omega. 
\]
Moreover, from \Cref{la:wiggleiter} (property \eqref{it:Calphaclose}), by using the gluing argument in \cite[Section 2.1]{KMSX24} and by choosing $\tilde{\eps}_i$ small enough, we can ensure
\[
 \|w_N - w_{N+1}\|_{C^{\alpha}(\bar\Omega,\R^d)} < 2^{-N}\eta.
\]
This gives \ref{Calphaconv}.

For property \ref{dustdoesntseeXi}, observe that on $\Omega^N_{err}$ we have $\nabla w_{N+1} = \nabla w_N$, and thus $\nabla w_{N+1}(x) \not \in \bigcup_{j =N}^\infty \{X_j\} $ if $x \in \Omega^N_{err}$. On each $\Omega^{N+1}_{err,i}$ the dust lies close to $X_{N}$ and thus (again, since $\lim_{n \to \infty} |X_n| = \infty$ and $X_n$ are pairwise distinct) we can ensure, by taking $\tilde{\eps}_i$ small enough, that $\nabla w_{N+1}(x) \notin \bigcup_{j=N+1}^\infty \{X_j\}$ for a.e. $x \in \Omega^{N+1}_{err}$.

Observe that $\Omega^N_{err} \cap \Omega^N_{ind} = \emptyset$ by assumption, and thus by the previous observation $\Omega^{N+1}_{err} \cap \Omega^{N+1}_{ind}=\emptyset$. Moreover, since $w_{N+1} = w_N$ outside of $\Omega^N_{ind}$ we have from the inductive assumption that
\[
 \nabla w_{N}(x) \in U \quad \text{a.e. in } \Omega^N_{err}
\]
and from \Cref{la:wiggleiter} we know that 
\[
 \nabla w_i^{N+1}(x) \in U \quad \text{a.e. in } \Omega^N_{err,i}.
\]
We obtain property \ref{errorN}.

% and $X_{N+1} \not \ni K \supset \supp \omega_N$. choosing $\tilde{\eps}_i$ suitably close, we can ensure that $\nabla w_i^{N+1}$ is either close to $X_n$ or $\nabla w_{i}^{N+1} \in \supp \omega_N \subset K \not \ni X_{n+1}$. Since $X_{n+1} \neq X_n$ and $\supp \omega_N$ are finitely many points, we can thus ensure that $\nabla w(x)\neq X_{N+1})$ for a.e. $x \in \Omega_{i}^N$, which ensures $\Omega_i^{N+1} \cap \

Moreover,
\[
\begin{split}
 \int_{\Omega^{N+1}_{err}} (1+|\nabla w_{N+1}|)^q 
 &= \int_{\Omega^N_{err}} (1+|\nabla w_N|)^q +\int_{\bigcup_{i=1}^{L_N}\Omega_{err,i}^N} (1+|\nabla w_{N+1}|)^q\\
&\le \eta |\Omega| (1-2^{-N}) + \sum_{i=1}^{L_N}\tilde{\eps}_i |\Omega_i^N| (1+|Y_{N+1}|)^q\\
&\leq\eta |\Omega| (1-2^{-N}) +\eta |\Omega| 2^{-N-1}
=\eta |\Omega| (1-2^{-N-1}),
 \end{split}
\]
where we used the inductive assumption in the second inequality and choose $\tilde{\eps}_i$ small enough, e.g., $\tilde{\eps}_i \le \eta2^{-N-1}(1+|Y_{N+1}|)^{-q}$ in \Cref{la:wiggleiter} on each $\Omega_i^N$. This gives \ref{errorN:int}.

We also obtain \ref{OmegaIndest} by choosing the $\tilde{\eps}_i$ small enough, and observing that $\nabla w_{N+1}(x) = X_{N+1}$ implies that $\nabla w_{N}(x) = X_{N}$, so we can apply the induction hypothesis.

In particular this implies 
\[
\begin{split}
 \frac{|\{x\in \Omega \setminus \Omega^{N+1}_{err}\colon \nabla w_{N+1}(x) \in X_{N+1}\}|}{|\{x\in \Omega_{ind}^N \colon \nabla w_{N}(x) \in X_{N} \cap \{X_{N+1}\} \}} 
\in \left( \frac{c_N}{c_{N+1}} \gamma_{N+1}, \frac{c_{N+1}}{c_N} \gamma_{N+1} \right)  
 \end{split}
\]
which gives, by induction,
\[
\begin{split}
 \frac{|\{x\in \Omega \setminus \Omega^{N+1}_{err}\colon \nabla w_{N+1}(x) \in X_{N+1}\}|}{|\Omega|} 
&\in \left( \frac{c_N}{c_{N+1}}\beta_{N+1},\frac{c_{N+1}}{c_N}\beta_{N+1} \right) \\
&= \left( \frac{c_N}{c_{N+1}}\gamma^{N+1}(\{X_{N+1}\}),\frac{c_{N+1}}{c_N}\gamma^{N+1}(\{X_{N+1}\}) \right).
 \end{split}
\]
This implies that for any Borel set $E \subset \R^{d \times m}$
\[ C_{N+1}^{-1} \nu^{N+1}(E \cap \{X_{N+1}\}) \leq \frac{|\{x \in \Omega \setminus \Omega_{err}^{N+1}\colon \nabla w(x) \in E \cap \{X_{N+1}\}\}|}{|\Omega|} \leq C_{N+1} \nu^{N+1}(E\cap \{X_{N+1}\}).
\]
Now observe that 
\[
\Omega \setminus \Omega^{N+1}_{err} =  \Omega \setminus \brac{\Omega_{ind}^N\cup \Omega^{N}_{err}} \dot{\cup} \bigcup_{i=1}^{L_N} \Omega_i^N \setminus \Omega_{err,i}^{N+1}
\]
and for two null sets $\mathcal{N}_1,\mathcal{N}_2$,
\[
 \Omega \setminus \brac{\Omega_{ind}^N\cup \Omega^{N}_{err}} \cup \mathcal{N}_1= \{x \in \Omega\setminus \Omega^N_{err}: \nabla w_{N+1}(x) = \nabla w_{N}(x) \neq X_{N}\} \cup \mathcal{N}_2
\]
Thus
\begin{equation}\label{eq:eeeee}
\begin{split}
 &|\{x\in \Omega \setminus \Omega^{N+1}_{err}\colon \nabla w_{N+1}(x) \in E\setminus \{X_{N+1}\}\}| \\
 % &= |\{x\in \Omega \setminus \brac{\Omega^N_{ind}\cup \Omega^{N+1}_{err}}\colon \nabla w_{N+1}(x) \in E\}| + |\{x\in \Omega^N_{ind}\setminus \Omega^{N+1}_{err}\colon \nabla w_{N+1}(x) \in E\}|\\
 &=|\{x\in \Omega\setminus \Omega^N_{err} : \nabla w_N(x)\in E\setminus\{X_N \cup X_{N+1}\}\}| + \sum_{i=1}^{L_N}|\{x\in \Omega^N_i\setminus \Omega^{N+1}_{err,i}\colon \nabla w_{N+1}(x) \in E \setminus \{X_{N+1}\} \}|.
 %&= |\{x\in \Omega\setminus \Omega^{N+1}_{err} \colon \nabla u_{N+1}(x) \in E\setminus \{X_N\}\}| + \sum_{i=1}^{L_N}|\{x\in \Omega_i^N\setminus \Omega^{N+1}_{err}\colon \nabla u_{N+1}(x) \in E\}|.
\end{split}
\end{equation}
For the second term above we can apply \eqref{eq:upperlowernuinftyOmegaiN}, for the first term we apply the induction hypothesis, to arrive at 
\[
\begin{split}
&\frac{|\{x\in \Omega \setminus \Omega^{N+1}_{err}\colon \nabla w_{N+1}(x) \in E\setminus \{X_{N+1}\}\}|}{|\Omega|}\\
\leq& c_N \nu^N(E \setminus \{X_N,X_{N+1}\}) + \frac{c_{N+1}}{c_N} \frac{|\Omega^N_{ind}|}{|\Omega|}\omega_{N+1}(E \setminus \{X_{N+1}\})\\
\leq&c_{N+1} \brac{\nu^N(E \setminus \{X_N,X_{N+1}\}) + \beta_N \omega_{N+1}(E \setminus \{X_{N+1}\})}.
 \end{split}
\]
In view of \eqref{eq:lfoinSL}
\[
\beta_{N} \omega_{N+1}(E \setminus \{X_{N+1}\}) + \nu^N (E \setminus \{X_N,X_{N+1}\})= \nu^{N+1}(E \setminus \{X_{N+1}\}) 
\]
and we have shown
\[
\begin{split}
&\frac{|\{x\in \Omega \setminus \Omega^{N+1}_{err}\colon \nabla w_{N+1}(x) \in E\setminus \{X_{N+1}\}\}|}{|\Omega|}\\
\leq& c_N \nu^N(E \setminus \{X_N,X_{N+1}\}) + \frac{c_{N+1}}{c_N} \frac{|\Omega^N_{ind}|}{|\Omega|}\omega_{N+1}(E \setminus \{X_{N+1}\})\\
\leq&c_{N+1} \nu^{N+1}(E \setminus \{X_{N+1}\}).
 \end{split}
\]
The lower bound follows analogously. Combining \eqref{eq:eeeee} and \eqref{eq:upperlowernuinfty} yields \ref{borelerrN}.

Once $w_N$ is constructed, the remainder of the proof is the verbatim as in 
\cite[Proposition~4.4]{KMSX24}, \eqref{asslam} ensures weak convergence in $W^{1,p-\eps}$ for $\eps > 0$, and the Cauchy condition on the H\"older norm ensures that the sequence $(w_N)$
converges uniformly to a map 
\[
w \in W^{1,1}(\Omega,\R^d)\cap C^{\alpha}(\overline{\Omega},\R^d),
\]
and that $w$ satisfies the required properties.
\end{proof}

\begin{proof}[Proof of \Cref{th:thm4.3}]
Fix $\eps$, $q$, $\alpha$ and $\Omega$ and $X_0$ as in \Cref{reduc}, and take the corresponding staircase laminate $\nu^\infty_{X_0}$ from the assumptions of \Cref{th:thm4.3}. W.l.o.g. $q > p$. 

Applying \Cref{pr:onestaircase}, we find $w$ that satisfies \eqref{eq2.3}, \eqref{eq2.1}, we only need to establish \eqref{eq2.2lower}.

Applying \eqref{eq:upperlowernuinfty} to $E \coloneqq  B(0,t)^c$ combined with \eqref{eq1.1} gives
\[
|\{x\in\Omega \setminus \Omega_{err}\colon |\nabla w(x)|>t\}|\leq M^p(1+|X_0|^p)|\Omega|t^{-p} \; \text{ for all }t>0.
\]
Moreover, by Chebyshev, since $q > p$,
\[
|\{x\in\Omega_{err}\colon |\nabla w(x)|>t\}| \aleq t^{-p} \int_{\Omega} |\nabla w|^p \leq \leq t^{-p} \int_{\Omega} (1+|\nabla w|)^q \leq \eps |\Omega|,
\]
for an $\eps>0$ of our choice, which readily leads  to \eqref{eq2.2lower}.

 If we have in addition \eqref{eq:1.1lower}, by \eqref{eq:upperlowernuinfty} in \Cref{pr:onestaircase}, we obtain \eqref{eq:exactreducibility} as follows
 \[
 \begin{split}
     |\{x\in\Omega : |\nabla w(x)|>t\}| &\ge |\{x\in\Omega\setminus\Omega_{err} : |\nabla w(x)|>t\}| \\
     &\ageq (1-\tilde{\varepsilon})m^p|\Omega|t^{-p} \\
     &\ge \frac{1}{2}m^p|\Omega|t^{-p}.
 \end{split}
 \]
We can conclude.
\end{proof}

\section{Application to the Astala-Faraco-Sz\'{e}kelyhidi result: Proof of Theorem~\ref{th:AFS}}\label{s:AFS}
%\iarmin{Rewrite this section: Make use of the Propositions of previous section}
%\iigor{I can do it but later.}
Throughout the remainder of the paper we fix $\la>1$. We introduce the \emph{good} set

\begin{equation}\label{def:KLambda}
 \mathscr{K}_{\Lambda} \coloneqq \left \{ \begin{pmatrix} a & b\\c & d \end{pmatrix} \in \R^{2 \times 2}\colon A\begin{pmatrix} a \\
                          b                                                               \end{pmatrix} + \begin{pmatrix} -d \\
                          c                                                               \end{pmatrix} = 0,  \text{ for some } A=\begin{pmatrix}
                              \lambda & 0\\
                              0 & \alpha
                              \end{pmatrix},
                          \lambda\in\{\la,\la^{-1}\}, \alpha\in[\la^{-1},\la] 
\right \}.
\end{equation}

As usual in these sort of arguments the reason we are interested in this set is the observation that if  $w \coloneqq  (u,v)$ satisfies for some $x \in \B^2$
\[
\nabla w(x) \equiv \begin{pmatrix} \partial_1 u(x) & \partial_2 u\\ \partial_1 v & \partial_2 v\end{pmatrix} \in \mathscr{K}_{\Lambda}
\]
then there exists a diagonal $A(x)\coloneqq \begin{pmatrix}
                              \lambda(x) & 0\\
                              0 & \alpha(x)
                              \end{pmatrix}$ such that $\frac{1}{\Lambda} I \leq A \leq \Lambda I$, and 
\begin{equation}\label{kappaimpliespde}
A(x)\nabla u(x) + \nabla^\perp v(x) = 0.
\end{equation}
If this happens for a.e. $x \in \B^2$ we can take the divergence and arrive at the PDE
\[
\div(A \nabla u) = 0 \quad \text{in $\B^2$}.
\]
So, our goal is to find $w: \B^2 \to \R^2$ with a prescribed boundary data at $\partial \B^2$ so that $\nabla w \in \mathscr{K}_{\Lambda}$ a.e. in $\B^2$ ---  and show that we can do so with certain properties on the growth of $w$.                            
The main new point is that if we choose correctly the boundary condition $u(x) = l$ on $\partial \B^2$, then we can ensure that the standard convex integration scheme lies entirely in the following set 
\begin{equation}\label{def:reasonableset}
\mathscr{U}_{\bar{a},\delta,\gamma} \coloneqq \left \{ \begin{pmatrix} a & b\\ -c & d\end{pmatrix} \in \R^{2 \times 2}\colon \quad a,d > \bar{a},\  \abs{1-\frac{a}{d}} < \delta, \text{ and } |1-b|,|1-c|< \gamma\right \}
\end{equation}  
for some $\bar{a},\delta,\gamma>0$. It is clear that $\mathscr{U}_{\bar{a},\delta,\gamma}$ is an open subset of $\R^{2 \times 2}$.

For technical convenience we also define the (not-open) set
\begin{equation}\label{def:reasonableset0}
\mathscr{U}_{\bar{a},0,\gamma} \coloneqq \left \{ \begin{pmatrix} a & b\\ -c & a\end{pmatrix} \in \R^{2 \times 2}\colon \quad a > \bar{a}, \quad |1-b|,|1-c|< \gamma\right \}
\end{equation}

and for $\Gamma > 1$
\[
\mathscr{K}_{\Lambda,\Gamma} \coloneqq  \left \{ \begin{pmatrix} a & b\\c & d \end{pmatrix} \in \mathscr{K}_{\Lambda}\colon \Gamma^{-1}\leq \frac{|a|+|b|}{|c| + |d|} \leq \Gamma
\right \}.
\]

We briefly discuss our overall strategy. Firstly, we show that whenever $\delta$ and $\gamma$ are not too large, and $\bar{a}$ is not too small, and $X \in\mathscr{U}_{\bar{a},\delta,\gamma}$ 
then we find a staircase laminate $\nu^\infty \in SL(\mathcal{K}_\Lambda,\mathscr{U}_{\bar{a},\delta,\gamma})$, i.e., supported on $\mathscr{K}_\Lambda$ with construction steps in $\mathscr{U}_{\bar{a},\delta,\gamma}$ with barycenter $X$ (see \Cref{la:laminate}) --- we fix those choices of $\delta,\gamma,\bar{a}$ throughout the rest of the argument.

Then, when we run the high–oscillatory perturbation (i.e., wiggle) lemma, \Cref{pr:onestaircase}, starting from $X \in \mathscr{U}_{\bar{a},\delta,\gamma}$ 
we can ensure that the resulting $\nabla w$ belongs to the  error set $\Omega_{err}$ (the set where we do not satisfy the equation $\nabla u \not \in \mathscr{K}_{\Lambda}$) and stays in $\mathscr{U}_{\bar{a},\delta,\gamma}$ 
a.e. (in other words the ``dust'' is in $\mathscr{U}_{\bar{a},\delta,\gamma}$). Thanks to this, we can restart the process. This is the (only) new idea. The advantage of only considering laminates in $\mathscr{U}_{\bar{a},\delta,\gamma}$ is that we can obtain better integrability estimates than if we were to work with $\R^{2\times 2}$.

Secondly, since we are in $\mathscr{U}_{\bar{a},\delta,\gamma}$ 
(not in a generic subset of $\R^{2 \times2}$) we can restart the oscillatory perturbation procedure from $\mathscr{U}_{\bar{a},\delta,\gamma}$, 
meaning we can obtain a better convergence rate. 
%This allows us to omit the original arguments used in \cite{KMSX24} concerning a \emph{reduction in weak $L^p$} method (see \Cref{reduc}). 

Here is the main result of this section
\begin{theorem}\label{th:UbardeltagammaReducible}
For any $\Lambda > 1$ there exists $\delta,\gamma,\bar{a},\Gamma > 0$ such that $\mathcal{U}_{\bar{a},\delta,\gamma}$ can be $\mathcal{U}_{\bar{a},\delta,\gamma}$-reduced to $\mathscr{K}_{\Lambda,\Gamma}$ exactly in weak $L^p$ for $p = \frac{2\Lambda}{\Lambda-1}$.
\end{theorem}

\Cref{th:UbardeltagammaReducible} is a consequence of the following \Cref{la:laminate} and the laminate criterion \Cref{th:thm4.3} proven in the previous section. It is clear that $\delta, \gamma,\bar{a}$ as below exist and can easily be computed.

% later in the convex integration scheme and is analogous to the construction in \cite[Proposition~3.1]{KMSX24} (with the additional requirement that all elementary splittings remain inside $\mathscr{U}_{\bar a,\delta,\gamma}$).
\begin{proposition}\label{la:laminate}
Let $\Lambda > 1$, $\Gamma > 2\Lambda+10$ and $\delta,\gamma \in (0,\frac{1}{2})$, $\bar{a} > 4$ be such that the following holds
\begin{equation}\label{eq:lam:cond1}
 1-\delta > \frac{1}{\Lambda},
\end{equation}
\begin{equation}\label{eq:lam:cond2}
\gamma < \frac{\Lambda-1}{\Lambda+1},
\end{equation}
\begin{equation}\label{eq:lam:cond3} 
\bar{a} > \frac{\Lambda}{\Lambda-1}+10, 
\end{equation}
\begin{equation}\label{eq:lam:cond4}
\frac{1}{\bar{a}} < \delta.
\end{equation}

Set $p=\frac{2\la}{\la-1}$. There exists a constant $M=M(\Lambda)>0$, and $m = m(\bar{a},\Lambda)$ with the following properties.

For any $X_0 \in \mathscr{U}_{\bar{a},\delta,\gamma}$ there exists $\nu^\infty\in SL(\mathscr{K}_{\Lambda,\Gamma},\mathscr{U}_{\bar{a},\delta,\gamma})$ 
a staircase laminate with barycenter $X_0$, cf. \Cref{def:SLKU}, satisfying all of the following
\begin{enumerate}
    \item \label{lam:prop4}  for all $t>0$
\begin{equation}\label{lam:upper}
\nu^\infty(\{X\colon  |X| > t\}) \leq M\, (1+|X_0|^p) t^{-p};
\end{equation}
    \item \label{lam:prop5} for all $t > 1$ 
    \begin{equation}\label{lam:lower}
 \nu^\infty(\{X\colon  |X| > t\}) \geq m\, t^{-p}.
\end{equation}

\end{enumerate}
\end{proposition}

\begin{proof}
\textsc{Step 1. Reduction to equal diagonal entries.}
Fix $X_{0} \in \mathscr{U}_{\bar{a},\delta,\gamma}$, we write
\[
X_{0} \coloneqq \begin{pmatrix}
      x & \sigma_1\\
      -\sigma_2 & y
     \end{pmatrix} \in\Ka.
\]
Assume first \underline{$x\neq y$}

Using the definition of $\Ka$ 
and the assumptions we obtain:
\[
x,y>\bar{a}>0, \quad  \abs{1-\frac{x}{y}}<\delta<1-\frac{1}{\la}, \quad \text{ and }\abs{1-\sigma_1},\abs{1-\sigma_2}<\gamma<\frac{\la-1}{\la+1}.
\]
Our first elementary splitting is to ensure $x=y$ in all subsequent iterations.

\textsc{Case 1: }Assume firstly that $x < y$. Observe that $\frac{1}{\Lambda} y < x$ and thus $\frac{\la(y-x)}{(\la-1)y}\in(0,1)$.

Thus the following is a convex rank one decomposition
\[
\begin{split}
X_0=
\begin{pmatrix}
      x & \sigma_1\\
      -\sigma_2 & y
     \end{pmatrix} 
     &= \frac{\Lambda (y-x)}{(\Lambda -1)y} \begin{pmatrix}
      \frac{1}{\Lambda} y & \sigma_1\\
      -\sigma_2 & y
     \end{pmatrix}  + \brac{1-\frac{\Lambda (y-x)}{(\Lambda -1)y}} \begin{pmatrix}
      y & \sigma_1\\
      -\sigma_2 & y
     \end{pmatrix}\\
     &\eqqcolon \frac{\Lambda (y-x)}{(\Lambda -1)y} G_{\frac{1}{2}} + \brac{1-\frac{\Lambda (y-x)}{(\Lambda -1)y}} X_{\frac{1}{2}}.
\end{split}
     \]
We claim that $G_{\frac{1}{2}} \in \mathscr{K}_{\Lambda,\Gamma}$ and $X_{\frac{1}{2}} \in \mathscr{U}_{\bar{a},\delta,\gamma}$. The latter is obvious since $X_0 \in \mathscr{U}_{\bar{a},\delta,\gamma}$, indeed $X_{\frac{1}{2}} \in \mathscr{U}_{\bar{a},0,\gamma}$.

Secondly, take the matrix $A = \begin{pmatrix} \Lambda & 0\\ 0 & \frac{\sigma_2}{\sigma_1} \end{pmatrix}$.
We claim $\frac{1}{\la} \le \frac{\sigma_2}{\sigma_1}\le \la$. Indeed, since $X_0\in\Ka$ 
and by \eqref{eq:lam:cond2} we obtain both inequalities
\[
\frac{1}{\la} \le \frac{1-\gamma}{1+\gamma} \le \frac{\sigma_2}{\sigma_1} \le \frac{1+\gamma}{1-\gamma} \le \la.
\]
Moreover, since $|y| > 1$, this yields $G_{\frac{1}{2}}\in\mathscr{K}_{\Lambda,\Gamma}$.

\textsc{Case 2:} $x > y$. In this case we decompose differently
\[
\begin{split}
 X_0=
\begin{pmatrix}
      x & \sigma_1\\
      -\sigma_2 & y
     \end{pmatrix} 
     &= \frac{\Lambda (x-y)}{(\Lambda -1)x} \begin{pmatrix}
      x & \sigma_1\\
      -\sigma_2 & \frac{1}{\Lambda} x
     \end{pmatrix}  + \brac{1-\frac{\Lambda (x-y)}{(\Lambda -1)x}} \begin{pmatrix}
      x & \sigma_1\\
      -\sigma_2 & x
     \end{pmatrix}\\
     &\eqqcolon \frac{\Lambda (x-y)}{(\Lambda -1)x} G_{\frac{1}{2}} + \brac{1-\frac{\Lambda (x-y)}{(\Lambda -1)x}} X_{\frac{1}{2}}.
\end{split}
     \]
As previously we want to ensure that $\frac{\la(x-y)}{(\la-1)x}\in(0,1)$, $X_{\frac{1}{2}}\in\Ka$ 
and $G_{\frac{1}{2}}\in\mathscr{K}_{\Lambda,\Gamma}$.

Expression $\frac{\la(x-y)}{(\la-1)x}$ is clearly nonnegative as $x>y$. Observe that the required inequality $\frac{\la(x-y)}{(\la-1)x}<1$ is equivalent to $\la y>x$. We know that $1-\frac{x}{y}>-\delta$ and by \eqref{eq:lam:cond1} $-\delta>\frac{1}{\la}-1$. This gives $x<2y-\frac{y}{\la}$ (remembering $y>\bar{a}>0$). Thusly, it suffices to show that $2y-\frac{y}{\la}\leq\la y$. However, this inequality is certainly true as it reduces to $(\la-1)^2>0$, valid for $\la\neq 1$.
By taking a matrix $A = \begin{pmatrix} \frac{1}{\Lambda} & 0\\ 0 & \frac{\sigma_2}{\sigma_1} \end{pmatrix}$ and proceeding analogously as before we get $X_{\frac{1}{2}} \in \mathscr{U}_{\bar{a},0,\gamma} \subset \mathscr{U}_{\bar{a},\delta,\gamma}$ and $G_{\frac{1}{2}}\in\mathscr{K}_{\Lambda,\Gamma}$.

\textsc{Case 3} If $x=y$ we simply skip this step and assume $\delta = 0$.

\textsc{Step 2. The basic iterative scheme.}

From now on we assume $x=y$. The following is a rank-1 decomposition
\begin{equation}\label{eq:Xdecomposition}
\begin{split}
X_0 \coloneqq \begin{pmatrix}
      x & \sigma_1\\
      -\sigma_2 & x
     \end{pmatrix} 
     &= \frac{\Lambda}{\Lambda (x + 1)-x} \begin{pmatrix}
      \frac{1}{\Lambda} x & \sigma_1\\
      -\sigma_2 & x
     \end{pmatrix}  + \brac{1-\frac{\Lambda}{\Lambda (x + 1)-x}} \begin{pmatrix}
      x+1 & \sigma_1\\
      -\sigma_2 & x
     \end{pmatrix}\\
     &=\lambda_{1,1}G_{1,1} + (1-\lambda_{1,1})\begin{pmatrix}
      x+1 & \sigma_1\\
      -\sigma_2 & x
     \end{pmatrix}.
     \end{split}
\end{equation}
Clearly $\lambda_{1,1}=\frac{\Lambda}{\Lambda (x + 1)-x} \in (0,1)$ for any $x>0$ and $\la>1$.

Furthermore, choosing in \eqref{def:KLambda}
\[
 A \coloneqq\begin{pmatrix}
       \Lambda & 0\\
       0 & \frac{\sigma_2}{\sigma_1}
      \end{pmatrix}
\]
we see that, proceeding as in \textsc{Step 1},
\begin{equation}\label{eq:G11}
 G_{1,1} = \begin{pmatrix}
      \frac{1}{\Lambda} x & \sigma_1\\
      -\sigma_2 & x
     \end{pmatrix} \in \mathscr{K}_{\Lambda,\Gamma}.
\end{equation}

%\iarmin{important for construction}
Moreover, we have
\[
\begin{pmatrix}
      x+1 & \sigma_1\\
      -\sigma_2 & x
     \end{pmatrix} \in \mathscr{U}_{\bar{a},\delta,\gamma}.
\]
Indeed, clearly $x+1>x > \bar{a}$, and $\sigma_1,\sigma_2$ did not change. Using \eqref{eq:lam:cond4} we get 
\[
\abs{ 1-\frac{x+1}{x}} = \frac{1}{x} < \frac{1}{\bar{a}} < \delta.
\]

Next, we continue the decomposition of the second matrix of \eqref{eq:Xdecomposition}
\begin{equation}\label{eq:Xseconddecomposition}
\begin{split}
&\begin{pmatrix}
      x+1 & \sigma_1\\
      -\sigma_2 & x
     \end{pmatrix}\\
     &=  \frac{\Lambda}{\Lambda(x+1)-(x+1)} \begin{pmatrix}
      x+1 & \sigma_1\\
      -\sigma_2 & \frac{1}{\Lambda} (x+1)
     \end{pmatrix}  + \brac{1- \frac{\Lambda}{\Lambda(x+1)-(x+1)}} \begin{pmatrix}
      x+1 & \sigma_1\\
      -\sigma_2 & x+1
     \end{pmatrix}\\
     &\eqqcolon  \frac{\Lambda}{\Lambda(x+1)-(x+1)} G_{1,2} + \brac{1- \frac{\Lambda}{\Lambda(x+1)-(x+1)}} X_1.
\end{split}
     \end{equation}

Here we observe that $\frac{\la}{\Lambda(x+1)-(x+1)}< 1$: by \eqref{eq:lam:cond3} we have $x>\bar{a}>\frac{1}{\la-1}$. Multiplying firstly by $\la-1$ and then adding $\la-1$ to both sides we obtain 
\[                                                                                                                                                                                                                                                                
\la<\la(x+1)-(x+1).  
\]
Again taking in \eqref{def:KLambda} $A = \begin{pmatrix}
       \frac{1}{\Lambda} & 0\\
       0 & \frac{\sigma_2}{\sigma_1}
      \end{pmatrix}$ we can ensure that the first matrix  on the right-hand side of \eqref{eq:Xseconddecomposition} belongs to $\mathscr{K}_{\Lambda,\Gamma}$, i.e., 
\[
G_{1,2} = \begin{pmatrix}
      x+1 & \sigma_1\\
      -\sigma_2 & \frac{1}{\Lambda} (x+1)
     \end{pmatrix} \in \mathscr{K}_{\Lambda,\Gamma}.\]
As for the second matrix on the right-hand side of \eqref{eq:Xseconddecomposition} we have
\[
X_1 =  \begin{pmatrix}
      x+1 & \sigma_1\\
      -\sigma_2 & x+1
     \end{pmatrix} \in \mathscr{U}_{\bar{a},0,\gamma} \subset \mathscr{U}_{\bar{a},\delta,\gamma}.
\]
Indeed, it follows since $x+1>x>\bar{a}$ and $\sigma_1$ and $\sigma_2$ did not change.

Combining \eqref{eq:Xdecomposition} with \eqref{eq:Xseconddecomposition} we obtain
\begin{equation}\label{eq:inductionbegining}
 X_0 = \lambda_{1,1} G_{1,1} + \lambda_{1,2}G_{1,2} + \gamma_1 X_1,
\end{equation}
where $\lambda_{1,2}=\brac{1-\frac{\Lambda}{\Lambda (x + 1)-x}}\frac{\Lambda}{\Lambda(x+1)-(x+1)}$, $\gamma_1=\brac{1-\frac{\Lambda}{\Lambda (x + 1)-x}}\brac{1- \frac{\Lambda}{\Lambda(x+1)-(x+1)}}$.

\textsc{Step 3. Iterative scheme.}

We can iterate the construction from Step 2 by setting
\[
 X_{n-1} = \lambda_{n,1}G_{n,1}+\lambda_{n,2}G_{n,2}+\gamma_{n}X_n
\]
for
\begin{equation}\label{eq:Gndefs}
 G_{n,1} \coloneqq \begin{pmatrix}
      \frac{1}{\Lambda} (x+n-1) & \sigma_1\\
      -\sigma_2 & (x+n-1)
     \end{pmatrix} \in \mathscr{K}_{\Lambda,\Gamma}, \quad \lambda_{n,1} \coloneqq \frac{\Lambda}{\Lambda (x + n)-(x+n-1)},
\end{equation}
\[
 G_{n,2} \coloneqq\begin{pmatrix}
      x+n & \sigma_1\\
      -\sigma_2 & \frac{1}{\Lambda} (x+n)
     \end{pmatrix} \in \mathscr{K}_{\Lambda,\Gamma}, \quad \lambda_{n,2} \coloneqq \brac{1-\frac{\Lambda}{\Lambda (x + n)-(x+n-1)}}\frac{\Lambda}{\Lambda(x+n)-(x+n)},
\]
and for the soaring sequence
\begin{equation}\label{eq:Xn}
\begin{split}
 X_n &\coloneqq \begin{pmatrix}
      x+n & \sigma_1\\
      -\sigma_2 & x+n
     \end{pmatrix} \in \mathscr{U}_{\bar{a},0,\gamma} \subset \mathscr{U}_{\bar{a},\delta,\gamma}
\end{split}
     \end{equation}
     with
     \[
      \gamma_n \coloneqq \brac{1-\frac{\Lambda}{\Lambda (x + n)-(x+n-1)}}\brac{1-\frac{\Lambda}{\Lambda(x+n)-(x+n)}}.
     \]
Then $|X_n|$ is monotonically increasing with $\lim_{n \to \infty} |X_n| = \infty$.

Moreover, for later use, it is worth pointing out, since $x > \bar{a} \geq 4$ and $\abs{\sigma_1},\abs{\sigma_2} \leq 2$ 
\begin{equation}\label{eq:Xnest}
|X_n| \aeq |x+n| \quad \forall n =0,1,\ldots
\end{equation}
with constants independent of $x$ and $n$.

We also have that all elementary splittings above originate from a matrix in $\mathscr{U}_{\bar{a},\delta,\gamma}$.

Moreover, using the notation of \cite[Proposition 3.1]{KMSX24} we set
\begin{equation}\label{eq:AFSmundef}
\begin{split}
 \mu_n &\coloneqq \frac{\lambda_{n,1}}{\lambda_{n,1}+\lambda_{n,2}} \delta_{G_{n,1}} + \frac{\lambda_{n,2}}{\lambda_{n,1}+\lambda_{n,2}} \delta_{G_{n,2}}\\
 &= \frac{\lambda_{n,1}}{1-\gamma_n} \delta_{G_{n,1}} + \frac{\lambda_{n,2}}{1-\gamma_n} \delta_{G_{n,2}}
\end{split}
 \end{equation}
which is supported in $\mathscr{K}_{\Lambda,\Gamma}$ and
\[
 \omega_n = (1-\gamma_n) \mu_n + \gamma_n \delta_{X_n} = \lambda_{n,1} \delta_{G_{n,1}}  + \lambda_{n,2} \delta_{G_{n,2}}  + \gamma_{n} \delta_{X_n}
\]
is by our construction a laminate of finite order with barycenter $X_{n-1}$.
%Before we proceed further it is useful to note that 
%\[\frac{\lambda_{k,1}}{1-\gamma_k}=\frac{x+k}{2(x+k)-1},\quad \frac{\lambda_{k,2}}{1-\gamma_k}=\frac{x+k-1}{2(x+k)-1}.\]
Set
\[
 \beta_n \coloneqq \prod_{k=1}^n \gamma_k, \quad \beta_0 \coloneqq 1
\]
and define
\[
\nu^N \coloneqq \sum_{n=1}^N \beta_{n-1} (1-\gamma_n) \mu_n + \beta_N \delta_{X_N}
\]
which is a laminate of finite order with construction steps in $\mathscr{U}_{\bar{a},\delta,\gamma}$, with $\supp \nu^N \subset \mathscr{K}_{\Lambda,\Gamma} \cup \{X_N\}$, and its limit (in the sense of %\cite[Proposition 3.1]{KMSX24}
\Cref{defstair}) is a staircase laminate with construction steps in $\mathscr{U}_{\bar{a},\delta,\gamma}$ called $\nu^\infty$ with $\supp \nu^\infty \subset \mathscr{K}_{\Lambda,\Gamma}$ and barycenter $X_0$. This shows that $\nu^\infty\in SL(\mathscr{K}_{\Lambda,\Gamma},\mathscr{U}_{\bar{a},\delta,\gamma})$.
% since in \Cref{le:KUdisjoint} we have already established the condition $\Ka\cap\K=\emptyset$.

\textsc{Step 4. Further properties of the staircase laminate.}

Now we focus on proving \eqref{lam:upper} and \eqref{lam:lower}.

We want to apply \cite[Lemma 3.3]{KMSX24}. To do this, we first observe that
\begin{equation}\label{3.4}
|X_n|\leq |X_{n+1}|\leq c|X_n|\quad \text{for all } n =1,2,\ldots
\end{equation}
for some constant $c>1$ independent of $X_0$. Indeed, this is clear by construction, \eqref{eq:Xn}, observing that
\[
1 \leq \abs{\frac{x+(n+1)}{x+n}} \leq 2 \quad \forall x \geq 0,\ n =1,2,\ldots
\]

% \aksharav{fixing constants here} \akshara{($c$ could even depend on $\la$ %and on the norm $|X_0|$
% ). Clearly $|X_n|\leq |X_{n+1}|$ and by taking %$c\coloneqq|1+1/x|$ 
% $c=2$
% \[
% \frac{(x+n+1)^2}{(x+n)^2}=\bigg(1+\frac{1}{x+n}\bigg)^2\leq \bigg(1+\frac{1}{n}\bigg)^2 \leq 2^2.
% \]
% The latter implies $|X_{n+1}|<c|X_n|$, which gives \eqref{3.4}.}

Let us begin with the proof of \eqref{lam:prop4}, \eqref{lam:upper} --- the upper bound.

We need to find constants %$c_0(|X_0|,\la)=c_0\geq 1$ and $M_0(|X_0|,\la)=M_0\geq 1$ 
$c_0(\la)=c_0\geq 1$ and $M_0(\la)=M_0\geq 1$
such that for each $n$ we have  
\begin{equation}\label{supp:cond}
\supp\mu_n\subseteq\{X\colon |X|\leq c_0|X_n|\}
\end{equation}
and 
\begin{equation}\label{beta:cond}
\beta_n|X_n|^p\leq M_0 (1+|X_0|^p),
\end{equation}
for $p\coloneqq \frac{2\la}{\la-1}$.

The former, \eqref{supp:cond}, is obvious: $\supp \mu_n \subset \bigcup_{j=1}^n \bigcup_{i=1}^2 \{G_{j,i}\}$, we have a formula for the $G$, \eqref{eq:Gndefs}, and an estimate for $X_n$ \eqref{eq:Xnest}, and $\abs{\sigma_1},\abs{\sigma_2} \leq 2$.

We now discuss \eqref{beta:cond}
%We will show that for large $n$ we have $\beta_n\approx n^{-p}$, which ensures \eqref{beta:cond}, since $|X_n|^p\approx n^p$ (for small $n$ the existence of a constant $M_0$ is obvious, because there are only finitely many inequalities to satisfy).
%We have
%\[\beta_n=\prod_{j=1}^n \gamma_j=\prod_{j=1}^n\frac{x+j-1}{x+j}\cdot\frac{x+j-\frac{\la}{\la-1}}{x+j+\frac{1}{\la-1}}.\]
% This formula follows from the definition of $\gamma_j$ and a careful multiplication, i.e.,
% \begin{equation*}
% \begin{split}
% \gamma_j=&\brac{1-\frac{\Lambda}{\Lambda (x+j){-}(x+j-1)}}\brac{1-\frac{\Lambda}{\Lambda(x+j)-(x+j)}}\\
% =&\frac{\la(x+j)-(x+j-1)-\la}{\Lambda (x+j){-}(x+j-1)}\cdot\frac{\la(x+j)-(x+j)-\la}{\Lambda(x+j)-(x+j)}\\
% =&\frac{(x+j-1)(\la-1)}{(x+j)(\la-1)+1}\cdot\frac{(x+j)(\la-1)-\la}{(x+j)(\la-1)}=\frac{x+j-1}{x+j}\cdot\frac{x+j-\frac{\la}{\la-1}}{x+j+\frac{1}{\la-1}}.
% \end{split}
% \end{equation*}
%Now using $\prod_{j=1}^n(j+\theta)=\frac{\Gamma(n+1+\theta)}{\Gamma(1+\theta)}$ for any $\theta\in\R$ we can rewrite expression above as 
%\[\beta_n=\frac{x}{x+n}\frac{\Gamma(n+x-\frac{1}{\la-1})}{\Gamma(x-\frac{1}{\la-1})}\frac{\Gamma(x+\frac{\la}{\la-1})}{\Gamma(n+x+\frac{\la}{\la-1})}.\]
%By virtue of the asymptotics $\frac{\Gamma(n+\alpha)}{\Gamma(n+\beta)}\approx n^{\alpha-\beta}$ for large $n$ we obtain for large $n$ 
%\[\beta_n\approx n^{-1+x-\frac{1}{\la-1}-x-\frac{\la}{\la-1}}=n^{-\frac{2\la}{\la-1}}.\]
%This yields \eqref{beta:cond} for some constant $M_0$ dependent on $|X_0|$ (precisely only on $x$) and on $\la$.

We have
\[\beta_n=\prod_{j=1}^n \gamma_j=\prod_{j=1}^n\frac{x+j-1}{x+j}\cdot\frac{x+j-\frac{\la}{\la-1}}{x+j+\frac{1}{\la-1}}.\]

Recall that $x > \bar{a} > \frac{1}{\Lambda-1}+10$, so we have 
\[
\frac{x+j-1}{x+j}, \frac{x+j-\frac{\Lambda}{\Lambda-1}}{x+j}  \succsim_{\bar{a}} 1 \quad %\forall j\in\N, x>\bar{a} 
\text{for any } j \in \N \text{ and any } x > \bar{a}.
\]
Let $$a\coloneqq \frac{\Lambda}{\Lambda-1}, \;\;\;b\coloneqq  \frac{1}{\Lambda-1}.$$
Then we have for $n\geq 1$
\[\begin{split}
    \ln \prod_{j=1}^n \gamma_j&= \sum_{j=1}^n \left( \ln \frac{x+j-1}{x+j}+\ln \frac{x+j-a}{x+j+b} \right)\\
    &= \sum_{j=1}^n \left( \ln \left( 1- \frac{1}{x+j}\right)+ \ln \left(1- \frac{a+b}{x+j+b}\right) \right)\\
    &\leq \sum_{j=1}^n \left( -\frac{1}{x+j}-\frac{a+b}{x+j+b} \right)\\
    &\leq -(1+a+b)\sum_{j=1}^n \frac{1}{x+j+b}.
\end{split}\]
Note that 
\[
\begin{split}
    \sum_{j=1}^n \frac{1}{x+j+b}&\geq \sum_{j=1}^n \int_{x+j+b}^{x+j+b+1}\frac{1}{t}dt\\
    & \geq  \int_{x+b+1}^{x+b+n+1}\frac{1}{t}dt= \ln \frac{x+b+n+1}{x+b+1}.
\end{split}
\]
Therefore, we get
$$\ln \prod_{j=1}^n \gamma_j\leq -(1+a+b) \ln \frac{x+b+n+1}{x+b+1}.$$
Since $1+a+b=2\Lambda/(\Lambda-1)$, and using \eqref{eq:Xnest} we have 
%$$\beta_n\leq \prod_{j=1}^n \gamma_j\leq \left(\frac{x+b+1}{x+b+n+1}\right)^{\frac{2\Lambda}{\Lambda-1}} \asymp_{\Lambda} \brac{\frac{|X_0|}{|X_n|}}^p.$$
$$\beta_n\leq \prod_{j=1}^n \gamma_j\leq \left(\frac{x+b+1}{x+b+n+1}\right)^{\frac{2\Lambda}{\Lambda-1}} \asymp_{\Lambda,\bar{a}} \brac{\frac{|X_0|}{|X_n|}}^p.$$
This readily implies \eqref{beta:cond} -- and by \cite[Lemma 3.3]{KMSX24} we conclude \eqref{lam:upper}.

It remains to establish property 
%\eqref{lam:prop5}, 
\eqref{lam:lower}, again using \cite[Lemma 3.3]{KMSX24}. To do this, we must find constants $c_2=c_2(\la,\bar{a})>0$ and $M_1=M_1(\la,\bar{a})>0$ such that 
\begin{equation}\label{mun}
\mu_n(\{X\colon |X|\geq c_2|X_n|\})\geq c_2 \quad \forall n \in \N
\end{equation}
and 
\begin{equation}\label{beta:again}
\beta_n|X_n|^p\geq M_1 \quad \forall n \in \N.
\end{equation}

We first focus on the \eqref{mun}. Recalling the formula for $G_{n,i}$, \eqref{eq:Gndefs}, as for \eqref{eq:Xnest} we observe  that 
\[
|G_{n,1}| \aeq_{\Lambda,\bar{a}} |X_{n}| \quad \forall n \in \N, x > \bar{a}.
\]
Thus, recalling the formula for $\mu_n$, \eqref{eq:AFSmundef}, to prove \eqref{mun} it suffices to show 
\[
\frac{\lambda_{n,1}}{1-\gamma_n} \ageq_{\Lambda,\bar{a}} 1 \quad \forall n \in \N.
\]
But
\[\frac{\lambda_{n,1}}{1-\gamma_n}=\frac{x+n}{2(x+n)-1} \ageq_{\Lambda,\bar{a}} 1 \quad \forall x >\bar{a}, n \in \N.\]
This implies \eqref{mun}.

We now discuss \eqref{beta:again}. Recall
    \[
      \beta_n= \prod_{j=1}^n\gamma_j = \prod_{j=1}^n\brac{1-\frac{\Lambda}{\Lambda (x + j)-(x+j-1)}}\brac{1-\frac{\Lambda}{\Lambda(x+j)-(x+j)}}.
     \]
    Observe that for any $x > \bar{a}$ and any $j \in \N$ 
    \[
\frac{\Lambda}{\Lambda (x + j)-(x+j-1)} = \frac{1}{x+j} \frac{\Lambda}{\Lambda - \frac{x+j-1}{x+j}}\leq \frac{\Lambda}{\Lambda-1} \frac{1}{x+j},
    \]
    and 
\[
\frac{\Lambda}{\Lambda(x+j)-(x+j)} = \frac{\Lambda}{\Lambda-1} \frac{1}{x+j}.
\]
Thus,
\[
\beta_n \geq \prod_{j=1}^n \brac{1-\frac{\Lambda}{\Lambda-1} \frac{1}{x+j}}^2 \aeq \prod_{j=\lfloor x\rfloor }^n \brac{1-\frac{\Lambda}{\Lambda-1} \frac{1}{j}}^2 \aeq \brac{\frac{\lfloor x\rfloor}{n} }^{\frac{2\Lambda}{\Lambda-1}} 
\]
which implies 
\[
|X_n|^p \beta_n \ageq \brac{|x+n| \frac{\lfloor x\rfloor}{n} }^{p}  \ageq_{\bar{a},\Lambda} |x|^p \aeq |X_0|^p 
\]
for all $x > \bar{a}$, $n \in \N$. This is \eqref{beta:again}.

From \cite[Lemma 3.3]{KMSX24} we thus conclude 
\[
\nu^\infty (\{ X: |X| >t\}) \ageq |X_0|^p t^{-p} \quad \forall t \geq c_1 |X_0|
\]
and clearly for $t \in [1,c_1 |X_0|]$ we have
\[
\nu^\infty (\{ X: |X| >t\}) \geq \nu^\infty (\{ X: |X| >c_1 |X_0|\}) \ageq |X_0|^{-p} |X_0|^p =1.
\]
Together (since $|X_0| \ageq 1$) we have established \eqref{lam:lower}. We can conclude.
\qedhere

\end{proof}

    \subsection{Proof of Theorem~\ref{th:AFS}}
    We prove below a slightly more general statement from which
    Theorem~\ref{th:AFS} (more general in the sense of the boundary data) follows.
    ﻿
    \begin{theorem}
    Fix $\Lambda > 1$ and choose $\bar{a}, \delta,\gamma,\Gamma>0$ so that the assumptions \eqref{eq:lam:cond1}, \eqref{eq:lam:cond2}, \eqref{eq:lam:cond3}, \eqref{eq:lam:cond4} of \Cref{la:laminate} are all satisfied.

    Take any $b \in \R$ and any $\bar{X} \coloneqq  \begin{pmatrix}
    \bar{x}_0 & \bar{x}_1\\
    \bar{y}_0 & \bar{y}_1
    \end{pmatrix} \in \mathscr{U}_{\bar{a},\delta,\gamma}$, cf. \eqref{def:reasonableset}, and set $l(x)\coloneqq  \langle \begin{pmatrix} \bar{x}_0\\\bar{x}_1 \end{pmatrix}, x\rangle_{\R^2} + b$ for some $b \in \R$.
    ﻿
 Fix any $\alpha \in (0,1)$ and $\Omega\subset \R^2$ a regular domain. There exists $u \in W^{1,2}(\Omega, \R) \cap C^\alpha(\bar{\Omega}, \R)$, and a bounded, measurable diagonal matrix valued function $A(x) = \begin{pmatrix} \lambda_1(x) & 0\\
                                                                                                        0 & \lambda_2(x)
                                                                                                       \end{pmatrix}$ with $\frac{1}{\Lambda} \leq \lambda_1(x),\lambda_2(x) \leq \Lambda$ for a.e. $x\in\Omega$.
    with the following properties
    \begin{equation}\label{boundaryl}
    u(x) = l(x) \text{ for }x\in\partial \Omega;
    \end{equation}
    \begin{equation}\label{div}
    \div(A \nabla u) = 0 \text{ in a distributional sense in }\Omega;
    \end{equation}
    \begin{equation}\label{Lipsch}
    \nabla u \in L^{\frac{2\Lambda}{\Lambda-1},\infty}(\Omega);
    \end{equation}
    and in the sense of Lorentz spaces $L^{\frac{2\Lambda}{\Lambda-r},r}$ we have
    \begin{equation}\label{lorentz}
    \nabla u \not \in L^{\frac{2\Lambda}{\Lambda-1},r}(\Omega) \text{ for any } 0 < r < \infty,\text{ in particular } \nabla u \not \in L^{\frac{2\Lambda}{\Lambda-1}}(\Omega).
    \end{equation}
    \end{theorem}
    \begin{proof}

By \Cref{th:UbardeltagammaReducible}, $\mathscr{U}_{\bar{a},\delta,\gamma}$ can be $\mathscr{U}_{\bar{a},\delta,\gamma}$-reduced to $\mathscr{K}_{\Lambda,\Gamma}$ exactly in weak $L^p$, cf. \eqref{def:KLambda}.

    By \Cref{th:thm4.1}, this implies that for any $X_0 \coloneqq  \bar{X}\in \mathscr{U}_{\bar{a},\delta,\gamma}$ there exists $w = (u,v) \in W^{1,1}(\Omega,\R^2) \cap C^\alpha$, $\nabla w \in L^{\frac{2\Lambda}{\Lambda-1}}(\Omega)$ such that 
    \[
\nabla w \in \mathscr{K}_{\Lambda,\Gamma} \subset \mathscr{K}_{\Lambda} \quad \text{a.e. in $\Omega$}
    \]
    and 
    \[
w = X_0x + \begin{pmatrix} b\\0\end{pmatrix} \quad \text{on $\partial \Omega$}.
    \]
Clearly $u$ satisfies \eqref{boundaryl}, also \eqref{Lipsch} is an immediate consequence since $|\nabla u| \leq |\nabla w|$. 
    
Moreover, by construction of $\mathscr{K}_{\Lambda}$, \eqref{kappaimpliespde}, we find $A(x)$ with the desired properties and \eqref{div} is satisfied.   

The only remaining property to establish is \eqref{lorentz}. Observe that since $\nabla w \in \mathscr{K}_{\Lambda,\Gamma}$ we have 
\[
|\nabla w| \aeq_{\Gamma} |\nabla u| \aeq_{\Gamma} |\nabla v| \quad \text{a.e. in $\Omega$}.
\]
Thus, from \eqref{eq:wewant4.2c2} we find that for some constant $c_1,c_2 > 0$,
\[
\bigl|\{x\in\Omega\colon \ |\nabla u(x)|>t\}\bigr|
\geq \bigl|\{x\in\Omega\colon \ |\nabla w(x)|> c_1t\}\bigr|
\geq c_2 t^{-p} \quad \forall t \in (1,\infty).
\]

Thus for any $r < \infty$, $p\coloneqq \frac{2\Lambda}{\Lambda-1}$, 
\[
\begin{split}
\|\nabla u\|_{L^{p,r}(\Omega)}^r =& 
\int_0^\infty  \abs{\{x\in \Omega: |\nabla u(x)|>t\}}^{\frac{r}{p}} t^r \frac{dt}{t}\\
\ageq&\int_1^\infty  \abs{t^{-p}}^{\frac{r}{p}} t^r \frac{dt}{t}\\
=&+\infty
\end{split}
\]
and thus $\nabla u \not \in L^{p,r}(\Omega)$. In particular, $\nabla u \not \in L^{p}(\Omega) = L^{p,p}(\Omega)$. This finishes the proof.

    \end{proof}

    \bibliographystyle{abbrv}
    \bibliography{bib}
    ﻿
    ﻿
    \end{document}